\documentclass[12pt]{article}

\title{Beyond ECH capacities}
\author{Michael Hutchings\footnote{Partially supported by NSF grants DMS-1105820 and DMS-1406312.}}
\date{}

\usepackage{amssymb}
\usepackage{latexsym}
\usepackage{amsmath}
\usepackage{amsthm}
\usepackage{amscd}
\usepackage[dvips]{graphics}
\usepackage{overpic}
\usepackage[mathscr]{euscript}

\newcommand{\mc}[1]{{\mathcal #1}}

\numberwithin{equation}{section}

\newtheorem{theorem}{Theorem}[section]
\newtheorem{proposition}[theorem]{Proposition}
\newtheorem{corollary}[theorem]{Corollary}
\newtheorem{lemma}[theorem]{Lemma}
\newtheorem{lemma-definition}[theorem]{Lemma-Definition}

\newtheorem{conjecture}[theorem]{Conjecture}

\theoremstyle{definition}
\newtheorem{definition}[theorem]{Definition}
\newtheorem{remark}[theorem]{Remark}

\newtheorem{example}[theorem]{Example}

\newcommand{\ceil}[1]{\left\lceil #1 \right\rceil}

\newcommand{\C}{{\mathbb C}}

\newcommand{\R}{{\mathbb R}}
\newcommand{\N}{{\mathbb N}}
\newcommand{\Z}{{\mathbb Z}}

\newcommand{\op}{\operatorname}

\newcommand{\Ker}{\op{Ker}}

\newcommand{\current}{\mathscr{C}}

\newcommand{\bpm}{\begin{pmatrix}}
\newcommand{\epm}{\end{pmatrix}}

\renewcommand{\epsilon}{\varepsilon}

\begin{document}

\setcounter{tocdepth}{2}

\maketitle

\begin{abstract}
ECH (embedded contact homology) capacities give obstructions to symplectically embedding one four-dimensional symplectic manifold with boundary into another. These obstructions are known to be sharp when the domain and target are ellipsoids (proved by McDuff), and more generally when the domain is a ``concave toric domain'' and the target is a ``convex toric domain'' (proved by Cristofaro-Gardiner). However ECH capacities often do not give sharp obstructions, for example in many cases when the domain is a polydisk. This paper uses more refined information from ECH to give stronger symplectic embedding obstructions when the domain is a polydisk, or more generally a convex toric domain. We use these new obstructions to reprove a result of Hind-Lisi on symplectic embeddings of a polydisk into a ball, and generalize this to obstruct some symplectic embeddings of a polydisk into an ellipsoid. We also obtain a new obstruction to symplectically embedding one polydisk into another, in particular proving the four-dimensional case of a conjecture of Schlenk. 
\end{abstract}

\section{Introduction}

\subsection{Some previous results}

This paper is concerned with the question of when one symplectic four-manifold with boundary can be symplectically embedded into another. An important class of examples of symplectic four-manifolds with boundary is constructed as follows: If $\Omega$ is a domain in the first quadrant of the plane, define the ``toric domain''
\[
X_\Omega = \left\{z\in\C^2\mid \left(\pi|z_1|^2,\pi|z_2|^2\right )\in\Omega\right\},
\]
with the restriction of the standard symplectic form
\begin{equation}
\label{eqn:omegastd}
\omega = \sum_{i=1}^2dx_idy_i
\end{equation}
on $\C^2$. For example, if $\Omega$ is the triangle with vertices $(0,0)$, $(a,0)$, and $(0,b)$, then $X_\Omega$ is the ellipsoid
\[
E(a,b) = \left\{z\in\C^2\;\bigg|\;\frac{\pi|z_1|^2}{a} + \frac{\pi|z_2|^2}{b} \le 1\right\}. 
\]
As a special case of this, we define the ball
\[
B(a) = E(a,a).
\]
If $\Omega$ is the rectangle with vertices $(0,0)$, $(a,0)$, $(0,b)$, and $(a,b)$, then $X_\Omega$ is the polydisk
\[
P(a,b) = \left\{z\in\C^2\;\big|\; \pi|z_1|^2\le a,\;\pi|z_2|^2\le b\right\}.
\]
It is already a quite subtle question when one four-dimensional ellipsoid or polydisk can be symplectically embedded into another.

In \cite{qech}, embedded contact homology (ECH) was used to define, for any symplectic four-manifold with boundary $X$, a sequence of real numbers
\[
0 = c_0(X) \le c_1(X) \le c_2(X) \le \cdots \le \infty
\]
called {\em ECH capacities\/}, such that if $X$ symplectically embeds into $X'$ then
\[
c_k(X) \le c_k(X')
\]
for all $k$. For example, by \cite[Prop.\ 1.2]{qech}, the ECH capacities of an ellipsoid are given by
\begin{equation}
\label{eqn:ckEab}
c_k(E(a,b)) = N_k(a,b)
\end{equation}
where $(N_0(a,b),N_1(a,b),\ldots)$ denotes the sequence of all nonnegative integer linear combinations of $a$ and $b$, that is real numbers $ma+nb$ where $m,n\in\N$, arranged in nondecreasing order (with repetitions). By \cite[Thm.\ 1.4]{qech}, the ECH capacities of a polydisk are given by
\begin{equation}
\label{eqn:ckPab}
c_k(P(a,b)) = \min \left\{am+bn \;\big|\; (m,n)\in\N^2,\;(m+1)(n+1)\ge k+1 \right\}.
\end{equation}
For more computations of ECH capacities see e.g.\ \cite[Thm.\ 4.14]{bn} and \cite[Thm.\ 1.21]{concave}.

McDuff \cite{mcd} showed that the open ellipsoid $\op{int}(E(a,b))$ symplectically embeds into $E(c,d)$ if and only if $N_k(a,b)\le N_k(c,d)$ for all $k$. Thus ECH capacities give a sharp obstruction to symplectically embedding one four-dimensional ellipsoid into another. (It is still a subtle number-theoretic problem to decide whether this embedding criterion holds for any given $a,b,c,d$, see e.g.\ \cite{ms,ck}.) A similar argument \cite[Cor.\ 11]{pnas}, using a result of Frenkel-M\"uller \cite[Prop.\ 1.4]{fm}, shows that ECH capacities give a sharp obstruction to symplectically embedding an ellipsoid into a polydisk in four dimensions.

More generally, Cristofaro-Gardiner \cite{dcg} has shown that ECH capacities give a sharp obstruction to symplectically embedding a ``concave toric domain'' into a ``convex toric domain''. Here we use the following terminology:

\begin{definition}
A toric domain $X_\Omega$ is {\em convex\/} if
\begin{equation}
\label{eqn:Omega}
\Omega = \{(x,y)\mid 0\le x\le A,\;  0\le y \le f(x)\}
\end{equation}
where $f:[0,A]\to \R^{\ge 0}$ is a nonincreasing\footnote{The result in \cite{dcg} applies to a more general notation of ``convex toric domain'' in which one only assumes that $f$ is concave and $f(0)>0$.} concave function. The toric domain $X_\Omega$ is {\em concave\/}  if $\Omega$ is given by \eqref{eqn:Omega} where $f:[0,A]\to\R^{\ge 0}$ is a convex function with $f(A)=0$.
\end{definition}

For example, a polydisk is a convex toric domain, where $f$ is constant. A toric domain is both convex and concave if and only if it is an ellipsoid, in which case $f$ is linear.

It turns out that ECH capacities sometimes do not give very good obstructions to symplectic embeddings of a convex toric domain into another symplectic manifold, such as a polydisk $P(a,1)$ into a ball $B(c)$. For example, ECH capacities imply that if $P(2,1)$ symplectically embeds into $B(c)$ 
then $c\ge 2$. However Hind-Lisi \cite{hl} showed that in fact, if $P(2,1)$ symplectically embeds into $B(c)$, then $c\ge 3$. Note that the converse is also true, because $P(a,1)$ trivially symplectically embeds into $B(a+1)$ by inclusion.

\begin{remark}
\label{rem:folding}
The Hind-Lisi result is optimal, in the sense that $2$ is the largest value of $a$ such that $P(a,1)$ symplectically embeds into $B(c)$ if and only if $c\ge a+1$. If $a>2$, then ``symplectic folding'' can be used to symplectically embed $P(a,1)$ into $B(c)$ whenever $c>2+a/2$, see \cite[Prop.\ 4.3.9]{schlenk}. When $a>6$, multiple symplectic folding can be used to construct even better symplecting embeddings of $P(a,1)$ into a ball (loc.\ cit.).
\end{remark}

\subsection{Some new results}

In this paper we introduce a new way to obstruct certain four-dimensional symplectic embeddings, using embedded contact homology in a more refined way than ECH capacities. While this method can be employed in a variety of situations, for concreteness we focus here on the problem of symplectically embedding one convex toric domain into another. In Theorem~\ref{thm:convex} below, we show that if such a symplectic embedding exists, then a certain combinatorial criterion must hold. The statement of this criterion is a bit complicated and postponed to \S\ref{sec:statement}. First, here are some applications of Theorem~\ref{thm:convex}.

To start, we can reprove the Hind-Lisi result and extend it to obstruct symplectic embeddings of some other polydisks into balls:

\begin{theorem}
\label{thm:polydiskball}
Let $a\ge 1$ and suppose that the polydisk $P(a,1)$ symplectically embeds into the ball $B(c)$. Then:
\begin{itemize}
\item If $1\le a\le 2$ then $c\ge a+1$.
\item If $2\le a\le 4$ then $c\ge (10+a)/4$.
\item If $4\le a\le 9/2$ then $c\ge 7/2$.
\item If $9/2\le a\le 7$ then $c\ge (13+a)/5$.
\item If $7\le a\le 8$ then $c\ge 4$.
\end{itemize}
\end{theorem}

Of course, the third and fifth bullet points follow trivially from the ones above them. We have included the fifth bullet point because the bound $c\ge 4$ is significant up until $a=8$, at which point it ties the volume constraint
\[
a = \op{vol}(P(a,1)) \le \op{vol}(B(c)) = \frac{c^2}{2}.
\]
(The volume of $X_\Omega$ is the area of $\Omega$.)

Some improvement of Theorem~\ref{thm:polydiskball} is possible when $a>2$. For example, the following result shows that if $2\le a\le 12/5$, then one cannot do better than (the limit of) symplectic folding, see Remark~\ref{rem:folding}.

\begin{theorem}
\label{thm:folding}
If $2\le a\le 12/5$, and if $P(a,1)$ symplectically embeds into $B(c)$, then $c\ge 2+a/2$.
\end{theorem}

More calculation remains to be done to optimize the lower bound on $c$ when $a>12/5$.

We also obtain a sharp obstruction to symplectic embeddings of certain polydisks into integral ellipsoids. Note that $P(a,1)$ trivially embeds into $E(bc,c)$ by inclusion whenever $a+b\le bc$. In some cases the converse is true:

\begin{theorem}
\label{thm:polydiskellipsoid}
Let $1\le a\le 2$ and let $b$ be a positive integer. Then the polydisk $P(a,1)$ symplectically embeds into the ellipsoid $E(bc,c)$ if and only if $a+b\le bc$.
\end{theorem}

We also obtain a sharp obstruction to certain symplectic embeddings of one polydisk into another. Suppose that $a\ge b>0$ and $a'\ge b'>0$. Observe that $P(a,b)$ trivially symplectically embeds into $P(a',b')$ if $a\le a'$ and $b\le b'$. One can ask when the converse holds.
The first ECH capacity of $P(a,b)$, along with many other symplectic capacities, equals $b$. Thus if $P(a,b)$ symplectically embeds into $P(a',b')$, then $b \le b'$. In some cases we can also show that $a\le a'$:

\begin{theorem}
\label{thm:polydisk}
Let $a,b,c$ be real numbers with $a,b\ge 1$ and $c > 0$. Suppose that $P(a,1)$ symplectically embeds into $P(bc,c)$. Assume that
\begin{equation}
\label{eqn:polydiskassumption}
\frac{2b}{a}\ge 1 + \frac{b-1}{4\ceil{b}-1}.
\end{equation}
Then $a\le bc$.
\end{theorem}

For example, when $b=1$, we obtain:

\begin{corollary}
\label{cor:schlenkold}
If $1\le a \le 2$ and if $P(a,1)$ symplectically embeds into $P(c,c)$, then $a\le c$.
\end{corollary}

Corollary~\ref{cor:schlenkold} is the four-dimensional case of \cite[Conj.\ 3.10]{schlenkold}. Note that if $a>2$, then one can use symplectic folding to show that $P(a,1)$ symplectically embeds into $P(c,c)$ whenever $c>1+a/2$, see \cite[Prop.\ 4.4.4]{schlenk}. For $a>4$, better symplectic embeddings are possible, see \cite[Fig.\ 7.2]{schlenk}.

\begin{remark}
Assuming a technical conjecture, we can weaken the assumption \eqref{eqn:polydiskassumption} to $a\le 2b$, see Remark~\ref{rem:improve2}. The resulting improved version of Theorem~\ref{thm:polydisk} can then be restated as follows: If $P(a,b)$ symplectically embeds into $P(a',b')$ with $a\ge b$ and $a'\ge b'$, and if
\[
\frac{a}{b}\le 2 \frac{a'}{b'},
\]
then
$a\le a'$.
\end{remark}

The above theorems are just a few simple applications of Theorem~\ref{thm:convex}. Much remains to be explored to see what Theorem~\ref{thm:convex} implies about other symplectic embeddings.

\subsection{Symplectic embeddings of convex toric domains}
\label{sec:statement}

We now prepare to state Theorem~\ref{thm:convex}, which gives a general obstruction to symplectically embedding one four-dimensional convex toric domain into another. Some of the following definitions are analogues of related definitions in \cite[\S1.6]{concave} for concave toric domains.

\begin{definition}
\label{def:cip}
A {\em convex integral path\/} is a path $\Lambda$ in the plane such that:
\begin{itemize}
\item The endpoints of $\Lambda$ are $(0,y(\Lambda))$ and $(x(\Lambda),0)$ where $x(\Lambda)$ and $y(\Lambda)$ are nonnegative integers.
\item
$\Lambda$ is the graph of a piecewise linear concave function $f:[0,x(\Lambda)]\to[0,y(\Lambda)]$ with $f'(0)\le 0$, possibly together with a vertical line segment at the right.
\item The vertices of $\Lambda$ (the points at which its slope changes) are lattice points.
\end{itemize}
\end{definition}

\begin{definition}
A {\em convex generator\/} is a convex integral path $\Lambda$ such that:
\begin{itemize}
\item Each edge of $\Lambda$ (line segment between vertices) is labeled `$e$' or `$h$'.
\item Horizontal and vertical edges can only be labeled `$e$'.
\end{itemize}
\end{definition}

The following notation for convex generators is useful: If $a$ and $b$ are relatively prime nonnegative integers, and if $m$ is a positive integer, then $e_{a,b}^m$ denotes an edge whose displacement vector is $(ma,-mb)$, labeled `$e$'; $h_{a,b}$ denotes an edge with displacement vector $(a,-b)$, labeled `$h$'; and if $m>1$ then $e_{a,b}^{m-1}h_{a,b}$ denotes an edge with displacement vector $(ma,-mb)$, labeled `$h$'. A convex generator is then equivalent to a commutative formal product of symbols $e_{a,b}$ and $h_{a,b}$, where no factor $h_{a,b}$ may be repeated, and the symbols $h_{1,0}$ and $h_{0,1}$ may not be used. The equivalence sends a convex generator to the product over all of its edges of the corresponding factors.

The idea of the above definition is that, as we will see in \S\ref{sec:prove}, the boundary of any convex toric domain can be perturbed such that for its induced contact form, up to large symplectic action, the generators of the ECH chain complex correspond to convex generators. (The relevant notions from ECH are reviewed in \S\ref{sec:reviewECH} below.) Their ECH index and (approximate) symplectic action are then described as follows:

\begin{definition}
If $\Lambda$ is a convex generator, define its {\em ECH index\/} by
\begin{equation}
\label{eqn:ILambda}
I(\Lambda) = 2(L(\Lambda) -1) - h(\Lambda),
\end{equation}
where $L(\Lambda)$ denotes the number of lattice points in the region enclosed by $\Lambda$ and the axes (including lattice points on the boundary); and 
$h(\Lambda)$ denotes the number of edges of $\Lambda$ that are labeled `$h$'.
\end{definition}

\begin{example}
Every convex generator $\Lambda$ has $I(\Lambda)\ge 0$. The following are all the convex generators with $I\le 6$:
\begin{itemize}
\item The unique convex generator with $I=0$ is the formal product $1$. (The path $\Lambda$ has no edges and starts and ends at $(0,0)$.) There are no convex generators with $I=1$.
\item $I=2$: $e_{1,0}$ and $e_{0,1}$.
\item $I=3$: $h_{1,1}$.
\item $I=4$: $e_{1,0}^2$, $e_{1,1}$, and $e_{0,1}^2$.
\item $I=5$: $h_{2,1}$ and $h_{1,2}$.
\item $I=6$: $e_{1,0}^3$, $e_{2,1}$, $e_{1,0}e_{0,1}$, $e_{1,2}$, and $e_{0,1}^3$.
\end{itemize}
\end{example}

\begin{definition}
If $\Lambda$ is a convex generator and $X_\Omega$ is a convex toric domain, define the {\em symplectic action\/} of $\Lambda$ with respect to $X_\Omega$ by
\begin{equation}
\label{eqn:ALambda}
A_\Omega(\Lambda) = A_{X_\Omega}(\Lambda) = \sum_{\nu\in \op{Edges}(\Lambda)}\vec{\nu}\times p_{\Omega,\nu}.
\end{equation}
Here, if $\nu$ is an edge of $\Lambda$, then $\vec{\nu}$ denotes the vector given by the lower right endpoint of $\nu$ minus the upper left endpoint.
Also $p_{\Omega,\nu}$ denotes a point on the tangent line\footnote{Here a ``tangent line'' to $\partial\Omega$ is a line $\eta$ touching $\partial\Omega$ such that all of $\Omega$ is in the closed half plane to the lower left of $\eta$.} to $\partial\Omega$ parallel to $\nu$. Finally, $\times$ denotes the determinant of a pair of 2-component vectors.
\end{definition}

\begin{example}
\label{ex:action}
\begin{itemize}
\item
If $X_\Omega$ is the polydisk $P(a,b)$, then
\[
A_{P(a,b)}(\Lambda) = bx(\Lambda) + ay(\Lambda).
\]
\item
If $X_\Omega$ is the ellipsoid $E(a,b)$, then $A_{E(a,b)}(\Lambda)=c$, where the line $bx+ay=c$ is tangent to $\Lambda$.
\end{itemize}
\end{example}

\begin{definition}
Let $X_\Omega$ be a convex toric domain. We say that a convex generator $\Lambda$ with $I(\Lambda)=2k$ is {\em minimal\/} for $X_\Omega$ if:
\begin{description}
\item{(a)}
All edges of $\Lambda$ are labeled `$e$'.
\item{(b)}
 $\Lambda$ uniquely minimizes $A_\Omega$ among convex generators with $I=2k$ and all edges labeled `$e$' (or equivalently among convex integral paths with $L=k+1$).
\end{description}
\end{definition}

\begin{remark}
It follows from Proposition~\ref{prop:dan} below that if $I(\Lambda)=2k$ and $\Lambda$ is minimal for $X_\Omega$, then $A_\Omega(\Lambda)=c_k(X_\Omega)$.
\end{remark}

If $\Lambda$ is a convex generator, let $m(\Lambda)$ denote the total multiplicity of all edges of $\Lambda$, i.e.\ the number of lattice points on the path $\Lambda$ minus one, or equivalently the total exponent of all factors $e_{a,b}$ and $h_{a,b}$ in the corresponding formal product.

\begin{definition}
\label{def:le}
Let $\Lambda,\Lambda'$ be convex generators such that all edges of $\Lambda'$ are labeled `$e$', and let $X_\Omega,X_{\Omega'}$ be convex toric domains. We write $\Lambda\le_{X_\Omega,X_{\Omega'}}\Lambda'$, or $\Lambda\le_{\Omega,\Omega'}\Lambda'$ for short, if the following three conditions hold:
\begin{description}
\item{(i)} $I(\Lambda) = I(\Lambda')$.
\item{(ii)} $A_\Omega(\Lambda)\le A_{\Omega'}(\Lambda')$.
\item{(iii)} $x(\Lambda) + y(\Lambda) - h(\Lambda)/2 \ge x(\Lambda') + y(\Lambda') + m(\Lambda') - 1$.
\end{description}
\end{definition}

The idea of this definition is that, as we will see in \S\ref{sec:prove}, if $X_\Omega$ symplectically embeds into $X_{\Omega'}$, then in the resulting cobordism between their (perturbed) boundaries, $\Lambda\le_{\Omega,\Omega'}\Lambda'$ is a necessary condition for the existence of an embedded irreducible holomorphic curve with ECH index zero between the ECH generators corresponding to $\Lambda$ and $\Lambda'$. The inequality (iii) is the key ingredient that allows us to go beyond ECH capacities, and arises from the fact that every holomorphic curve must have nonnegative genus, cf.\ Proposition~\ref{prop:J0}.

\begin{definition}
  Let $\Lambda_1$ and $\Lambda_2$ be convex generators. We say that $\Lambda_1$ and $\Lambda_2$ ``have no elliptic orbit in common'' if, when we write $\Lambda_1$ and $\Lambda_2$ as formal products, no factor $e_{a,b}$ appears in both $\Lambda_1$ and $\Lambda_2$.  Similarly, we say that $\Lambda_1$ and $\Lambda_2$ ``have no hyperbolic orbit in common'' if no factor $h_{a,b}$ appears in both of the formal products corresponding to $\Lambda_1$ and $\Lambda_2$.  \end{definition}

If $\Lambda_1$ and $\Lambda_2$ are convex generators with no hyperbolic orbit in common, then we define their ``product'' $\Lambda_1\Lambda_2$ 
by concatenating the formal products of symbols $e_{a,b}$ and $h_{a,b}$ corresponding to $\Lambda_1$ and $\Lambda_2$. This product operation on convex generators is associative (when defined).

\begin{theorem}
\label{thm:convex}
Let $X_\Omega$ and $X_{\Omega'}$ be convex toric domains. Suppose there exists a symplectic embedding $X_\Omega\to X_{\Omega'}$. Let $\Lambda'$ be a convex generator which is minimal for $X_{\Omega'}$.
Then there exist a convex generator $\Lambda$ with $I(\Lambda)=I(\Lambda')$, a nonnegative integer $n$, and product decompositions $\Lambda=\Lambda_1\cdots\Lambda_n$ and $\Lambda'=\Lambda_1'\cdots\Lambda_n'$, such that:
\begin{itemize}
\item
$
\Lambda_i\le_{\Omega,\Omega'}\Lambda_i'
$
for each $i=1,\ldots,n$.
\item
Given $i,j\in\{1,\ldots,n\}$, if $\Lambda_i\neq \Lambda_j$ or $\Lambda_i'\neq\Lambda_j'$, then $\Lambda_i$ and $\Lambda_j$ have no elliptic orbit in common.
\item
If $S$ is any subset of $\{1,\ldots,n\}$, then $I\left(\prod_{i\in S}\Lambda_i\right) = I\left(\prod_{i\in S}\Lambda_i'\right)$.
\end{itemize}
\end{theorem}

\begin{remark}
\label{rem:improve}
We expect that in Theorem~\ref{thm:convex}, instead of assuming that $\Lambda'$ is minimal for $X_{\Omega'}$, it is enough to assume only that all edges of $\Lambda'$ are labeled `$e$'. As explained in Appendix~\ref{app:conjecture}, this would follow from a conjectural description of the differential on the ECH chain complex for the (perturbed) boundaries of convex toric domains.
\end{remark}

\begin{remark}
For most of the applications in this paper (Theorems~\ref{thm:polydiskball}, \ref{thm:polydiskellipsoid} and \ref{thm:polydisk}), we do not need the last two bullets in Theorem~\ref{thm:convex}.
\end{remark}

\begin{remark}
A corollary of Theorem~\ref{thm:convex} is that under the assumptions of the theorem, there exists a convex generator $\Lambda$ with $I(\Lambda)=I(\Lambda')$, $A_\Omega(\Lambda)\le A_{\Omega'}(\Lambda')$, and $x(\Lambda)+y(\Lambda) \ge x(\Lambda') + y(\Lambda')$. The resulting symplectic embedding obstructions are sometimes stronger than those given by ECH capacities, but are generally weaker than the obstructions obtained using the full force of Theorem~\ref{thm:convex}.
\end{remark}

\begin{remark}
\label{rem:variants}
Variants of Theorem~\ref{thm:convex} hold for symplectic embeddings of convex or concave toric domains into concave toric domains. We do not have a version of Theorem~\ref{thm:convex} for symplectic embeddings of concave toric domains into convex toric domains, but this is exactly the case where Cristofaro-Gardiner has shown that ECH capacities already give sharp symplectic embedding obstructions.
\end{remark}

The idea of the proof of Theorem~\ref{thm:convex} is that if $X_\Omega$ symplectically embeds into the interior of $X_{\Omega'}$, then there is a ``weakly exact'' symplectic cobordism from the boundary of $X_{\Omega'}$ to the boundary of $X_\Omega$ (when these boundaries are smooth). After perturbing these boundaries to nondegenerate contact three-manifolds, there is a cobordism map on embedded contact homology (ECH). Nontriviality of this cobordism map implies the existence of certain holomorphic curves in the (completed) cobordism. The key new ingredient here is that the particular cobordism above satisfies a ``tameness'' condition which allows us to control the holomorphic curves that might arise and rule out certain troublesome multiple covers. Carefully studying these holomorphic curves and encoding relevant topological information about them combinatorially then leads to the conclusions of Theorem~\ref{thm:convex}.

The main reason why the above cobordism satisfies the tameness condition is that the short elliptic Reeb orbits in the perturbed boundary of $X_\Omega$ have slightly positive rotation angle, see \S\ref{sec:deftame}. We would also have tameness if the short elliptic orbits in the perturbed boundary of $X_{\Omega'}$ had slightly negative rotation angle, which is why Theorem~\ref{thm:convex} has analogues for symplectic embeddings into concave toric domains, as mentioned in Remark~\ref{rem:variants}.

The rest of this paper is organized as follows. In \S\ref{sec:applying} we use Theorem~\ref{thm:convex} and some combinatorial calculations to deduce the applications in Theorems~\ref{thm:polydiskball}, \ref{thm:folding}, \ref{thm:polydiskellipsoid}, and \ref{thm:polydisk}. After this, it just remains to prove Theorem~\ref{thm:convex}. In \S\ref{sec:reviewECH} we review the relevant aspects of ECH. In \S\ref{sec:newECH} we introduce the crucial notion of ``$L$-tame'' symplectic cobordisms, and we prove a new result about ECH (Proposition~\ref{prop:key}) which controls the behavior of multiply covered holomorphic curves in such cobordisms. In \S\ref{sec:perturb} we study the ECH of the boundary of a (perturbed) convex domain in detail. In \S\ref{sec:prove} we put all of the above together to prove Theorem~\ref{thm:convex}. Finally, Appendix~\ref{app:conjecture} briefly discusses a conjectural improvement of Theorem~\ref{thm:convex}.

\section{Calculations using the main theorem}
\label{sec:applying}

We now use Theorem~\ref{thm:convex} as a ``black box'' to deduce Theorems~\ref{thm:polydiskball}, \ref{thm:folding}, \ref{thm:polydiskellipsoid}, and \ref{thm:polydisk} on symplectic embeddings of polydisks. The proof of Theorem~\ref{thm:convex} begins in \S\ref{sec:reviewECH}.

\subsection{Minimal convex generators for ellipsoids and polydisks}

We first need to clarify when the hypotheses of Theorem~\ref{thm:convex} are satisfied.

\begin{lemma}
\label{lem:action}
Fix $a,b>0$.
\begin{description}
\item{(a)} Let $\eta$ be a line of slope $-b/a$ which goes through a lattice point $(x,y)$ in the first quadrant. Let $\Lambda$ be the maximal convex integral path for which $\eta$ is a tangent line, with all edges labeled `$e$'. Then $\Lambda$ is minimal for $E(a,b)$.
\item{(b)}
Let $x,y$ be nonnegative integers. Suppose that
\[
bx'+ay'>bx+ay
\]
whenever $(x',y')$ is a different pair of nonnegative integers with 
\[
(x'+1)(y'+1)\ge (x+1)(y+1).
\]
Then $\Lambda = e_{1,0}^xe_{0,1}^y$ is minimal for $P(a,b)$.
\end{description}
\end{lemma}

\begin{proof}
(a) This is similar to \cite[Ex.\ 1.23]{concave}. 
If $\Lambda'$ is any other convex generator with $L(\Lambda')=L(\Lambda)$ and all edges labeled `$e$', let $\eta'$ be the tangent line to $\Lambda'$ of slope $-b/a$.
Then $\eta'$ must be to the upper right of (and not equal to) $\eta$, since otherwise we would have $L(\Lambda)\le L(\Lambda')$, with equality only if $\Lambda'=\Lambda$. It now follows from Example~\ref{ex:action} that $A_{E(a,b)}(\Lambda')>A_{E(a,b)}(\Lambda)$.

(b) We have $L(\Lambda)=k+1$ where $k=(x+1)(y+1)-1$. If $\Lambda'$ is any other convex generator with all edges labeled `$e$' and $L(\Lambda')=k+1$, then since $\Lambda'$ is contained in the rectangle $[0,x(\Lambda')]\times [0,y(\Lambda')]$, it follows that $(x(\Lambda')+1)(y(\Lambda')+1) \ge k+1$, with equality only if $\Lambda'=e_{1,0}^{x(\Lambda')}e_{0,1}^{y(\Lambda')}$. So by the hypothesis and Example~\ref{ex:action}, we have $A_{P(a,b)}(\Lambda') \ge A_{P(a,b)}(\Lambda)$, with equality only if $(x,y)=(x',y')$, in which case $\Lambda'=\Lambda$.
\end{proof}

\subsection{Symplectic embeddings of polydisks into balls}

We now prove Theorem~\ref{thm:polydiskball}.

\begin{lemma}
Let $a\ge 1$ and suppose that $P(a,1)$ symplectically embeds into $B(c)$. Then for every positive integer $d$ we have
\begin{equation}
\label{eqn:y1bound}
c \ge \min\left(1+a,\frac{3d-2+a}{d},\frac{d+3}{2}\right).
\end{equation}
\end{lemma}

\begin{proof}
By Lemma~\ref{lem:action}(a), if $d$ is a positive integer then $e_{1,1}^d$ is minimal for $B(c)$. We can then apply Theorem~\ref{thm:convex} to $\Lambda'=e_{1,1}^d$. We do so in two steps. Below, the symbol `$\le$' between convex generators means `$\le_{P(a,1),B(c)}$'.

{\em Step 1.\/} Suppose that $\Lambda\le e_{1,1}^d$. Then by Definition~\ref{def:le} we have $I(\Lambda)=d(d+3)$ and $x(\Lambda)+y(\Lambda) \ge 3d-1$. If $y(\Lambda)>0$, then since $a\ge 1$, we have
\begin{equation}
\label{eqn:3d-2+a}
3d-2+a \le x(\Lambda) + ay(\Lambda) = A_{P(a,1)}(\Lambda) \le A_{B(c)}(e_{1,1}^d) = dc.
\end{equation}
If $y(\Lambda)=0$, then the only possibility is that $\Lambda=e_{1,0}^{d(d+3)/2}$, so
\[
\frac{d(d+3)}{2} = A_{P(a,1)}(\Lambda) \le A_{B(c)}(e_{1,1}^d) = dc.
\]
We conclude that if $\Lambda\le e_{1,1}^d$, then
\begin{equation}
\label{eqn:step1c}
c\ge \min\left(\frac{3d-2+a}{d},\frac{d+3}{2}\right),
\end{equation}
and if also $c < (3d-2+a)/d$ then $\Lambda = e_{1,0}^{d(d+3)/2}$.

{\em Step 2.\/} Let $d$ be any positive integer. We now apply Theorem~\ref{thm:convex} to $\Lambda'=e_{1,1}^d$ to obtain a convex generator $\Lambda$ with $I(\Lambda)=d(d+3)$, a positive integer $n$, and factorizations $\Lambda'=\Lambda_1'\cdots\Lambda_n'$ and $\Lambda=\Lambda_1\cdots\Lambda_n$, with $A_{P(a,1)}(\Lambda_i)\le A_{B(c)}(\Lambda_i')$ for each $i$. Since $\Lambda'$ has only one edge, it follows that $\Lambda_i'=e_{1,1}^{d_i}$ where $d_1,\ldots,d_n$ are positive integers with $\sum_{i=1}^nd_i=d$.

If we suppose that
\begin{equation}
\label{eqn:c3}
c < \frac{3d'-2+a}{d'} \quad\quad \forall d'\in\{1,\ldots,d-1\},
\end{equation}
then we must have $n=1$, so that $\Lambda\le e_{1,1}^{d}$. Otherwise, Step 1 implies that $\Lambda_i=e_{1,0}^{d_i(d_i+3)/2}$ for each $i=1,\ldots,n$. Then
\[
I(\Lambda) = \sum_{i=1}^n d_i(d_i+3) = 3d+\sum_{i=1}^nd_i^2 < 3d+ d^2,
\]
which is a contradiction.

So by Step 1, the inequalities \eqref{eqn:c3} imply \eqref{eqn:step1c}. Equivalently,
\begin{equation}
\label{eqn:step2clong}
c \ge \min\left(\min_{d'=1,\ldots,d}\frac{3d'-2+a}{d'},\frac{d+3}{2}
\right).
\end{equation}
Since the linear functions $a\mapsto (3d'-2+a)/d'$ for different $d'$ agree at $a=2$, and their slope is a decreasing function of $d'$, it follows that in the minimum in \eqref{eqn:step2clong} we can restrict attention to $d'=1,d$. Thus \eqref{eqn:step2clong} implies \eqref{eqn:y1bound}.
\end{proof}

\begin{proof}[Proof of Theorem~\ref{thm:polydiskball}.]
 We can read off the conclusions of Theorem~\ref{thm:convex} from the inequality \eqref{eqn:y1bound}. The case $d=4$ implies the first three conclusions. The case $d=5$ implies the fourth and fifth conclusions\footnote{The first conclusion of Theorem~\ref{thm:polydiskball} also follows from the inequality \eqref{eqn:y1bound} for $d=3$. The inequality \eqref{eqn:y1bound} for $d>5$ gives lower bounds on $c$ which are weaker than the volume constraint $c\ge \sqrt{2a}$.}.
\end{proof}

\subsection{Symplectic folding is sometimes best}

\begin{proof}[Proof of Theorem~\ref{thm:folding}.]
Assume that $2\le a \le 12/5$ and that $P(a,1)$ symplectically embeds into $B(c)$. Suppose that
\begin{equation}
\label{eqn:ca2}
c < 2 + a/2.
\end{equation}
We will obtain a contradiction in four steps. Below, the symbol `$\le$' between convex generators means `$\le_{P(a,1),B(c)}$'.

{\em Step 1.\/} We first show that if $\Lambda\le e_{1,1}^d$ with $d\le 9$, then $y(\Lambda)\le 1$.

If $y(\Lambda)\ge 2$, then as in \eqref{eqn:3d-2+a} we have
\[
3d-3+2a\le dc.
\]
Combining this with \eqref{eqn:ca2} gives
\[
2d - 6 < (d-4)a.
\]
If $1\le d\le 3$ then it follows that $a<4/3$; if $d=4$ then it follows that $2<0$; and if $5\le d \le 9$ then it follows that $a>12/5$. Either way, this contradicts our hypothesis that $2\le a \le 12/5$.

{\em Step 2.\/} We now show that if $d$ is a positive integer, if $\Lambda\le e_{1,1}^d$, and $y(\Lambda)\le 1$, then $\Lambda$ includes a factor of $e_{1,0}$.

If not, then the only possibility for $\Lambda$ with the correct ECH index is
\[
\Lambda = e_{(d^2+3d-2)/2,1}.
\]
The action inequality in the definition of $\le_{P(a,1),B(c)}$ then implies that
\[
\frac{d^2+3d-2}{2} + a \le dc.
\]
Combining this with \eqref{eqn:ca2} gives
\[
d^2-d-2 < (d-2)a.
\]
Similarly to Step 1, it follows that $a<2$ or $a>4$, contradicting our hypothesis.

{\em Step 3.\/} We now show that there does not exist any convex generator $\Lambda$ with $\Lambda\le e_{1,1}^9$.

If $\Lambda$ is such a generator, then we know from Step 1 that $y(\Lambda)\le 1$.

If $y(\Lambda)=0$, then the only possibility for $\Lambda$ with the correct ECH index is $\Lambda=e_{1,0}^{54}$. Then $54\le 9c$, which combined with \eqref{eqn:ca2} implies that $a>8$, contradicting our hypothesis.

If $y(\Lambda)=1$, then we must have $x(\Lambda)\ge 27$, or else we would have $I(\Lambda)\le 106$, contradicting the fact that $I(\Lambda)=108$. Since $x(\Lambda)\ge 27$, it follows that
\[
27 + a \le 9c.
\]
Combining this with \eqref{eqn:ca2} gives $a>18/7$, contradicting our hypothesis that $a\le 12/5$.

{\em Step 4.\/}
We now apply Theorem~\ref{thm:convex} to $\Lambda'=e_{1,1}^9$ to obtain a convex generator $\Lambda$, a positive integer $n$, and factorizations $\Lambda=\Lambda_1\cdots\Lambda_n$ and $\Lambda'=\Lambda_1'\cdots\Lambda_n'$ satisfying the three bullets in Theorem~\ref{thm:convex}.

By Step 3 and the first bullet, we must have $n>1$.

By Steps 1 and 2 and the first bullet, each $\Lambda_i$ contains a factor of $e_{1,0}$. Then by the second bullet, all of the $\Lambda_i$ must be equal, and all of the $\Lambda_i'$ must be equal. Thus either $n=9$ and $\Lambda_i'=e_{1,1}$ for each $i$, or $n=3$ and $\Lambda_i'=e_{1,1}^3$ for each $i$.

If $n=9$, then by Steps 1 and 2 we have $\Lambda=e_{1,0}^2$ for each $i$. But then $I(\Lambda)=36$, contradicting the fact that $I(\Lambda)=108$.

If $n=3$, then by Steps 1 and 2, and the facts that $I(\Lambda_i)=18$ and $x(\Lambda_i)+y(\Lambda_i)\ge 8$, the only possibilities are that $\Lambda_i=e_{1,0}^9$ for each $i$, or $\Lambda_i=e_{1,0}e_{6,1}$ for each $i$. In the former case we have $I(\Lambda)=54$, and in the latter case we have $I(\Lambda)=102$. Either way, this contradicts the fact that $I(\Lambda)=108$.
\end{proof}

\subsection{Symplectic embeddings of polydisks into ellipsoids}

\begin{proof}[Proof of Theorem~\ref{thm:polydiskellipsoid}.]
Let $a\ge 1$, let $b$ be a positive integer, suppose that $P(a,1)$ symplectically embeds into $E(bc,c)$, and assume that
\begin{equation}
\label{eqn:abbc}
a + b > bc.
\end{equation}
We need to show that $a>2$. If $b=1$ then this follows from Theorem~\ref{thm:polydiskball}, so we assume below that $b\ge 2$.

{\em Step 1.\/} We first show that if $\Lambda \le_{P(a,1),E(bc,c)} e_{b,1}$ then $\Lambda=e_{1,0}^{b+1}$.

We have $I(\Lambda) = 2b+2$ and $x(\Lambda) + y(\Lambda) \ge b+1$. If $y(\Lambda)>0$, then since $a\ge 1$ we have
\[
a+b \le ay(\Lambda) + x(\Lambda) = A_{P(a,1)}(\Lambda) \le A_{E(bc,c)}(e_{b,1}) = bc,
\]
contradicting \eqref{eqn:abbc}. Thus $y(\Lambda)=0$, so the only possibility is $\Lambda = e_{1,0}^{b+1}$.

{\em Step 2.\/} By Lemma~\ref{lem:action}, if $d$ is a positive integer then $e_{b,1}^d$ is minimal for $E(bc,c)$. We now apply Theorem~\ref{thm:convex} to $\Lambda'=e_{b,1}^2$ to obtain $\Lambda$ with $I(\Lambda) = 6b+4$, a positive integer $n$, and factorizations of $\Lambda$ and $\Lambda'$ into $n$ factors satisfying the first bullet of Theorem~\ref{thm:convex}. We must have $n=1$, since otherwise $n=2$ and $\Lambda_1'=\Lambda_2'=e_{b,1}$, and then Step 1 implies that $\Lambda_1=\Lambda_2 = e_{1,0}^{b+1}$, so $I(\Lambda)=4b+4$, contradicting the fact that $I(\Lambda) = 6b+4$.
Thus $\Lambda\le e_{b,1}^2$, and in particular $x(\Lambda) + y(\Lambda) \ge 2b+3$.

If $y(\Lambda)>0$, then
\[
2b+2+a \le A_{P(a,1)}(\Lambda) \le A_{E(bc,c)}(e_{b,1}^2)=2bc.
\]
It follows from this and \eqref{eqn:abbc} that $a>2$.

It $y(\Lambda)=0$, then $\Lambda=e_{1,0}^{3b+2}$, so
\[
3b+2 = A_{P(a,1)}(\Lambda) \le A_{E(bc,c)}(e_{b,1}^2) \le 2bc.
\]
By \eqref{eqn:abbc} it follows that $b+2<2a$. Since we are assuming that $b\ge 2$, we get $a>2$ again.
\end{proof}

\subsection{Symplectic embeddings of polydisks into polydisks}

\begin{proof}[Proof of Theorem~\ref{thm:polydisk}.]
Suppose that $P(a,1)$ symplectically embeds into $P(bc,c)$ where $a,b\ge 1$. Assume that \eqref{eqn:polydiskassumption} holds; in particular, $a\le 2b$. We want to show that $a\le bc$. Assume to get a contradiction that $a>bc$. We proceed in four steps. Below, the symbol `$\le$' between convex generators means `$\le_{P(a,1),P(bc,c)}$'.

{\em Step 1.\/} We first show that if $\Lambda'=e_{1,0}^d$ or $\Lambda'=e_{0,1}^d$ with $d>0$, and if $\Lambda\le \Lambda'$, then $d=1$ and $\Lambda=e_{1,0}$.

We have $x(\Lambda)+y(\Lambda) \ge 2d-1$, so $I(\Lambda)\ge 4d-2$. Since $I(\Lambda)=I(\Lambda')=2d$, this forces $d=1$. Then $I(\Lambda)=2$, so the only options are $\Lambda=e_{1,0}$ and $\Lambda=e_{0,1}$. The latter case is not possible because then $A_{P(a,1)}(\Lambda)=a$, but $A_{P(bc,c)}(\Lambda')\in\{c,bc\}$, and by assumption $a>bc\ge c$.

{\em Step 2.\/} We next show that if $\Lambda\le e_{1,0}^de_{0,1}^k$ with $d,k>0$, then $y(\Lambda) < k$.

We know that $x(\Lambda) + y(\Lambda) \ge 2d+2k-1$. If $y(\Lambda) \ge k$, then
\[
2d + k-1 + ka \le x(\Lambda) + ay(\Lambda) = A_{P(a,1)}(\Lambda) \le A_{P(bc,c)}(e_{1,0}^de_{0,1}^k) = c(d+kb).
\]
Since we are assuming that $c<a/b$, we deduce that $a > 2b + (k-1)b/d$, which contradicts our assumption that $a\le 2b$.

{\em Step 3.\/} A calculation using Lemma~\ref{lem:action}(b) shows that $\Lambda'=e_{1,0}^de_{0,1}^2$ is minimal for $P(bc,c)$ when $d=4\ceil{b}-2$. We now apply Theorem~\ref{thm:convex} to this $\Lambda'$ to obtain a convex generator $\Lambda$ with $I(\Lambda)=6d+4$, a positive integer $n$, and factorizations $\Lambda'=\Lambda'_1\cdots\Lambda'_n$ and $\Lambda=\Lambda_1\cdots\Lambda_n$, satisfying $\Lambda_i\le\Lambda_i'$ for each $i$. We claim that $n=1$.

Suppose to get a contradiction that $n>1$.

We first show that $y(\Lambda_i')\le 1$ for each $i$. Suppose to the contrary that for some $i=1,\ldots,n$ we have $\Lambda_i'=e_{1,0}^{d'}e_{0,1}^2$ with $d'<d$. By Step 1 we have $\Lambda_j = e_{1,0}$ for $j\neq i$. Then
\[
I(\Lambda) = 2(d-d')(y(\Lambda_i)+1) + I(\Lambda_i)
\]
Now $I(\Lambda_i)=I(\Lambda_i')=6d'+4$, and $y(\Lambda_i)\le 1$ by Step 2. Thus
\[
I(\Lambda) \le 4d+2d'+4 < 6d+4 = I(\Lambda),
\]
a contradiction.

Since $y(\Lambda_i')\le 1$ for each $i$, it follows from Steps 1 and 2 that $y(\Lambda_i)=0$ for each $i$. Thus
\[
I(\Lambda)=\sum_{i=1}^nI(\Lambda_i) = \sum_{i=1}^nI(\Lambda_i').
\]
On the other hand, if $X\in\{0,1,\ldots,d\}$ denotes the sum of $x$ of the two factors $\Lambda_i'$ that contain $e_{0,1}$, then
\[
\sum_{i=1}^n I(\Lambda_i') = 2d+2X+4 < 6d+4.
\]
Combining the above two lines gives a contradiction.

{\em Step 4.\/} We now complete the proof. Since $n=1$ in Step 3, we have $\Lambda\le \Lambda'=e_{1,0}^de_{0,1}^2$.

If $y(\Lambda)=0$, then $\Lambda=e_{1,0}^{3d+2}$, so
\[
3d+2 = A_{P(a,1)}(\Lambda) \le A_{P(bc,c)}(\Lambda') = c(d+2b).
\]
If $y(\Lambda)>0$, then since $x(\Lambda)+y(\Lambda) \ge 2d+3$, we have
\[
2d+2+a \le x(\Lambda) + ay(\Lambda) = A_{P(a,1)}(\Lambda) \le A_{P(bc,c)}(\Lambda') = c(d+2b).
\]
Either way, we have
\[
c(d+2b) \ge \min(3d+2,2d+2+a).
\]
Since $d=4\ceil{b}-2\ge 2b\ge a$, the above minimum is $2d+2+a$. Thus
\[
c\ge \frac{2d+2+a}{d+2b}.
\]
Since we are assuming that $c<a/b$, we obtain
\begin{equation}
\label{eqn:ad}
a > \frac{2b(d+1)}{d+b}.
\end{equation}
Plugging in $d=4\ceil{b}-2$ gives a contradiction to \eqref{eqn:polydiskassumption}.
\end{proof}

\begin{remark}
\label{rem:improve2}
If Theorem~\ref{thm:convex} can be improved as conjectured in Remark~\ref{rem:improve}, then in Steps 3 and 4 above we can take $d$ to be arbitrarily large. Then \eqref{eqn:ad} implies that $a\ge 2b$. This would allow the assumption \eqref{eqn:polydiskassumption} in Theorem~\ref{thm:polydisk} to be weakened to $a < 2b$ (and then to $a\le 2b$ by a simple argument).
\end{remark}

\section{Review of embedded contact homology}
\label{sec:reviewECH}

We now review those aspects of embedded contact homology that are needed in the proof of Theorem~\ref{thm:convex}. More details can be found in the survey \cite{bn}. The reader familiar with ECH may wish to skip ahead to \S\ref{sec:newECH}.

\subsection{Chain complex generators}

Let $Y$ be a closed oriented three-manifold, let $\lambda$ be a contact form on $Y$, and let $\Gamma\in H_1(Y)$. We now review how to define the embedded contact homology $ECH(Y,\lambda,\Gamma)$ with $\Z/2$ coefficients\footnote{One can also define ECH with integer coefficients, as explained in \cite[\S9]{obg2}, but that is not necessary for the applications here.}.

Let $\xi=\Ker(\lambda)$ denote the contact structure determined by $\lambda$, and let $R$ denote the Reeb vector field associated to $\lambda$.
A {\em Reeb orbit\/} is a map $\gamma:\R/T\Z\to Y$ for some $T>0$, modulo precomposition with translations, such that $\gamma'(t)=R(\gamma(t))$. Given a Reeb orbit $\gamma$, the linearization of the Reeb flow around $\gamma$ determines a symplectic linear map
\[
P_\gamma:(\xi_{\gamma(0)},d\lambda) \longrightarrow (\xi_{\gamma(0)},d\lambda).
\]
The Reeb orbit $\gamma$ is {\em nondegenerate\/} if $1\notin\op{Spec}(P_\gamma)$. We then say that $\gamma$ is {\em elliptic\/} if the eigenvalues of $P_\gamma$ are on the unit circle, {\em positive hyperbolic\/} if the eigenvalues of $P_\gamma$ are positive, and {\em negative hyperbolic\/} if the eigenvalues of $P_\gamma$ are negative.
Assume that $\lambda$ is nondegenerate, i.e.\ that all Reeb orbits are nondegenerate.

An {\em orbit set\/} is a finite set of pairs $\alpha=\{(\alpha_i,m_i)\}$ where the $\alpha_i$ are distinct embedded Reeb orbits and the $m_i$ are positive integers. We call $m_i$ the ``multiplicity'' of $\alpha_i$ in $\alpha$. We sometimes write an orbit set using the multiplicative notation $\alpha=\prod_i\alpha_i^{m_i}$. The homology class of the orbit set $\alpha$ is defined by
\[
[\alpha] = \sum_im_i[\alpha_i] \in H_1(Y).
\]
The orbit set $\alpha$ is {\em admissible\/} if $m_i=1$ whenever $\alpha_i$ is (positive or negative) hyperbolic. 

The embedded contact homology $ECH(Y,\lambda,\Gamma)$ is the homology of a chain complex $ECC(Y,\lambda,\Gamma,J)$, which is the free $\Z/2$-module generated by admissible orbit sets $\alpha$ with $[\alpha]=\Gamma$. The differential depends on a suitable almost complex structure $J$ on $\R\times Y$, and to define it we first need some more preliminaries.

\subsection{The ECH index}

If $\beta=\{(\beta_j,n_j)\}$ is another orbit set with $[\alpha]=[\beta]$, let $H_2(Y,\alpha,\beta)$ denote the set of $2$-chains $Z$ in $Y$ with $\partial Z = \sum_im_i\alpha_i - \sum_jn_j\beta_j$, modulo boundaries of $3$-chains. The set $H_2(Y,\alpha,\beta)$ is an affine space over $H_2(Y)$.

Given $Z\in H_2(Y,\alpha,\beta)$, the {\em ECH index\/} is an integer defined by
\begin{equation}
\label{eqn:defI}
I(\alpha,\beta,Z) = c_\tau(Z) + Q_\tau(Z) + CZ_\tau^I(\alpha) - CZ_\tau^I(\beta).
\end{equation}
Here $\tau$ is a homotopy class of trivializations of $\xi$ over the Reeb orbits $\alpha_i$ and $\beta_j$; $c_\tau(Z)$ denotes the relative first Chern class of $\xi$ over $Z$ with respect to $\tau$, see \cite[\S3.2]{bn}; $Q_\tau(Z)$ denotes the relative self-intersection number of $Z$ with respect to $\tau$, see \cite[\S3.3]{bn}; and
\[
CZ_\tau^I(\alpha) = \sum_i\sum_{k=1}^{m_i}CZ_\tau(\alpha_i^k),
\]
where $CZ_\tau$ denotes the Conley-Zehnder index with respect to $\tau$, and $\alpha_i^k$ denotes the $k$-fold cover of $\alpha_i$. The ECH index does not depend on the choice of trivialization $\tau$.


\subsection{Holomorphic curves and the index inequality}

We say that an almost complex structure $J$ on $\R\times Y$ is ``$\lambda$-compatible'' if $J(\partial_s)=R$, where $s$ denotes the $\R$ coordinate; $J(\xi)=\xi$; $d\lambda(v,Jv)\ge 0$ for $v\in\xi$; and $J$ is $\R$-invariant. Fix a $\lambda$-compatible $J$.
A ``$J$-holomorphic curve from $\alpha$ to $\beta$'' is a $J$-holomorphic curve in $\R\times Y$, where the domain is a possibly disconnected punctured compact Riemann surface, with positive ends asymptotic to covers $\alpha_i^{q_{i,k}}$ with total multiplicity $\sum_kq_{i,k}= m_i$, and negative ends asymptotic to covers $\beta_j^{q_{j,l}}$ with total multiplicity $\sum_lq_{j,l}= n_j$, see \cite[\S3.1]{bn}. A holomorphic curve $u$ as above determines a homology class $[u]\in H_2(Y,\alpha,\beta)$.

The {\em Fredholm index\/} of $u$ is defined by
\begin{equation}
\label{eqn:ind}
\op{ind}(u) = -\chi(u) + 2c_\tau(u) + CZ_\tau^{\op{ind}}(u).
\end{equation}
Here $\chi(u)$ denotes the Euler characteristic of the domain of $u$; $\tau$ is a homotopy class of trivialization of $\xi$ over $\alpha_i$ and $\beta_j$ as before; $c_\tau(u)$ is shorthand for $c_\tau([u])$;  and
\[
CZ_\tau^{\op{ind}}(u) = \sum_i\sum_kCZ_\tau(\alpha_i^{q_{i,k}}) - \sum_j\sum_lCZ_\tau(\beta_j^{q_{j,l}}).
\]
If $J$ is generic and $u$ has no multiply covered components, then the moduli space of $J$-holomorphic curves from $\alpha$ to $\beta$ is a manifold near $u$ of dimension $\op{ind}(u)$. Also, if $u$ has no multiply covered components, then without any genericity assumption on $J$, we have the index inequality
\begin{equation}
\label{eqn:indineq}
\op{ind}(u) \le I(u),
\end{equation}
where $I(u)$ is shorthand for $I(\alpha,\beta,[u])$, see \cite[\S3.4]{bn}.

\subsection{Holomorphic currents}

A ``$J$-holomorphic current from $\alpha$ to $\beta$'' is a finite formal sum $\current=\sum_kd_kC_k$ where the $C_k$ are distinct, irreducible (i.e.\ connected domain), somewhere injective $J$-holomorphic curves, such that if $C_k$ is a holomorphic curve from the orbit set $\alpha(k)$ to the orbit set $\beta(k)$, then $\alpha=\prod_k\alpha(k)^{d_k}$ and $\beta=\prod_k\beta(k)^{d_k}$. Here the ``product'' of two orbit sets is defined by adding the multiplicities of all Reeb orbits involved. The curves $C_k$ are the ``components'' of the holomorphic current $\current$, and the integers $d_k$ are the ``multiplicities'' of the components.

Let $\mc{M}^J(\alpha,\beta)$ denote the set of $J$-holomorphic currrents from $\alpha$ to $\beta$. Observe that $\R$ acts on this set by translation of the $\R$ coordinate on $\R\times Y$. Also, each $\current\in\mc{M}^J(\alpha,\beta)$ determines a homology class $[\current]\in H_2(Y,\alpha,\beta)$. Define the ECH index $I(\current) = I(\alpha,\beta,[\current])$.

The index inequality \eqref{eqn:indineq} can be used to show the following. Below, a {\em trivial cylinder\/} is a cylinder $\R\times\gamma\subset\R\times Y$ where $\gamma$ is an embedded Reeb orbit; this is automatically $J$-holomorphic.

\begin{proposition}
\label{prop:lowI}
(part of \cite[Prop.\ 3.7]{bn})
If $J$ is generic, then each $J$-holomorphic current $\current=\sum_id_iC_i$ in $\R\times Y$ has the following properties:
\begin{itemize}
\item $I(\current)\ge 0$.
\item $I(\current)=0$ if and only if each $C_i$ is a trivial cylinder.
\item If $I(\current)=1$, then one $C_i$ is embedded and has $\op{ind}=I=1$; and the curves $C_j$ for $j\neq i$ are trivial cylinders disjoint from $C_i$.
\end{itemize}
\end{proposition}

\subsection{The differential}

Given a generic $\lambda$-compatible $J$, we define the differential $\partial$ by
\[
\partial\alpha = \sum_\beta\sum_{Z\in H_2(Y,\alpha,\beta) : I(Z)=1}\#\frac{\mc{M}^J(\alpha,\beta,Z)}{\R}\cdot\beta
\]
where `$\#$' denotes the mod 2 count. It is shown for example in \cite[\S5.3]{bn} that $\partial$ is well-defined, and in \cite[\S7]{obg1} that $\partial^2=0$. We denote the homology of the chain complex $ECC(Y,\lambda,\Gamma,J)$ by $ECH(Y,\lambda,\Gamma,J)$.

If $\Gamma=0$ and $H_2(Y)=0$, then the chain complex $ECC(Y,\lambda,0,J)$ has a canonical $\Z$-grading, in which the grading of an admissible orbit set $\alpha$ is defined by
\begin{equation}
\label{eqn:absoluteZgrading}
I(\alpha) = I(\alpha,\emptyset,Z)
\end{equation}
where $Z$ is the unique element of $H_2(Y,\alpha,\emptyset)$. (For the grading in more general cases see \cite[\S3.6]{bn}.)

It follows from a theorem of Taubes \cite{e1}, identifying ECH with a version of Seiberg-Witten Floer cohomology \cite{km}, that $ECH(Y,\lambda,\Gamma,J)$ depends only on $Y$, the contact structure\footnote{In a certain sense, ECH does not depend on the contact structure either; see \cite[Rem.\ 1.7]{bn}.} $\xi=\Ker(\lambda)$, and $\Gamma$. This invariance of ECH currently cannot be proved directly by counting holomorphic curves, see \cite[\S5.5]{bn} and \S\ref{sec:newECH}.

An important example is that if $Y$ is diffeomorphic to $S^3$, and if $\Ker(\lambda)$ is the standard (fillable) contact structure\footnote{For a different contact structure, if one still uses the absolute grading \eqref{eqn:absoluteZgrading}, then \eqref{eqn:echs3} will hold with a grading shift.}, then in terms of the absolute grading \eqref{eqn:absoluteZgrading}, we have
\begin{equation}
\label{eqn:echs3}
ECH_*(Y,\lambda,0,J) = \left\{\begin{array}{cl}
\Z/2, & *=0,2,4,\ldots,\\
0, & \mbox{otherwise}.
\end{array}\right.
\end{equation}
Given the invariance of ECH, this can be proved by computing the ECH for a specific contact form as in \cite[\S3.7]{bn}.

\subsection{Topological complexity of holomorphic curves}

The index inequality \eqref{eqn:indineq} shows that the ECH index bounds the Fredholm index of holomorphic curves. Related to the ECH index is another topological quantity, denoted by $J_0$, which controls a certain sort of topological complexity of holomorphic curves. It is defined as follows: If $\alpha=\{(\alpha_i,m_i)\}$ and $\beta=\{(\beta_j,n_j)\}$ are orbit sets with $[\alpha]=[\beta]\in H_1(Y)$, and if $Z\in H_2(Y,\alpha,\beta)$, then
\begin{equation}
\label{eqn:defJ0}
J_0(\alpha,\beta,Z) = -c_\tau(Z) + Q_\tau(C) + CZ_\tau^J(\alpha) - CZ_\tau^J(\beta)
\end{equation}
where
\[
CZ_\tau^J(\alpha) = \sum_i\sum_{k=1}^{m_i-1}CZ_\tau(\alpha_i^k).
\]
The difference between the definition of $I$ in \eqref{eqn:defI}, and the definition of $J_0$ in \eqref{eqn:defJ0}, is that the sign of the Chern class term is switched, and the Conley-Zehnder term is slightly different. If $\current\in\mc{M}^J(\alpha,\beta)$, we write $J_0(\current)=J_0(\alpha,\beta,[\current])$.

The quantity $J_0$ controls topological complexity as follows. Suppose that $C\in\mc{M}^J(\alpha,\beta)$ is somewhere injective and irreducible. Let $g(C)$ denote the genus of $C$, let $n_i^+$ denote the number of positive ends of $C$ at covers of $\alpha_i$, and let $n_j^-$ denote the number of negative ends of $C$ at covers of $\beta_j$.

\begin{proposition}
\label{prop:J0}
Let $\alpha$ and $\beta$ be admissible orbit sets. Suppose that $C\in\mc{M}^J(\alpha,\beta)$ is somewhere injective and irreducible. Then
\begin{equation}
\label{eqn:tcb}
2g(C)-2 + \sum_i(2n_i^+-1) + \sum_j(2n_j^--1) \le J_0(C).
\end{equation}
\end{proposition}

\begin{proof}
This is a special case of \cite[Prop.\ 6.9]{ir}.
\end{proof}

\subsection{Filtered ECH}

If $\alpha=\{(\alpha_i,m_i)\}$ is an orbit set for $\lambda$, define its {\em symplectic action\/} $\mc{A}(\alpha)\in\R$ by
\[
\mc{A}(\alpha) = \sum_im_i\int_{\alpha_i}\lambda.
\]
It follows from our assumptions on the almost complex structure $J$ that if $\mc{M}^J(\alpha,\beta)$ is nonempty, then $\mc{A}(\alpha) \ge \mc{A}(\beta)$.

Given $L\in\R$, define
\[
ECC^L(Y,\lambda,\Gamma,J) \subset ECC(Y,\lambda,\Gamma,J)
\]
to be the span of the admissible orbit sets $\alpha$ satisfying $[\alpha]=\Gamma$ and $\mc{A}(\alpha)<L$. It follows from the above that this is a subcomplex. Its homology is the {\em filtered ECH\/}, denoted by $ECH^L(Y,\lambda,\Gamma)$. It is shown in \cite[Thm.\ 1.3]{cc2} that filtered ECH does not depend on the choice of almost complex structure $J$. However, unlike the usual (unfiltered) ECH, it does depend strongly on the contact form $\lambda$.

Inclusion of chain complexes induces maps
\[
ECH^L(Y,\lambda,\Gamma) \longrightarrow ECH(Y,\lambda,\Gamma)
\]
and
\[
ECH^L(Y,\lambda,\Gamma) \longrightarrow ECH^{L'}(Y,\lambda,\Gamma),
\]
for $L<L'$. It is shown in \cite[Thm.\ 1.3]{cc2} that these maps also do not depend on $J$.

\subsection{ECH capacities}

We now review how to define ECH capacities in the special case we need, namely for a compact smooth star-shaped domain $X\subset\R^4$, equipped with the standard symplectic form $\omega$ in \eqref{eqn:omegastd}. Here ``star-shaped'' means that the boundary $Y$ of $X$ is transverse to the radial vector field
\begin{equation}
\label{eqn:radialvf}
\rho = \frac{1}{2}\sum_{i=1}^2\left(x_i\partial_{x_i}+y_i\partial_{y_i}\right).
\end{equation}

The three-manifold $Y$ is diffeomorphic to $S^3$, and the $1$-form
\begin{equation}
\label{eqn:lambdastd}
\lambda_{std} = \frac{1}{2}\sum_{i=1}^2\left(x_idy_i - y_idx_i\right)
\end{equation}
restricts to a contact form on $Y$. For a nonnegative integer $k$, the $k^{th}$ ECH capacity of $X$ is defined by $c_k(X)=c_k(Y,\lambda)$, where $c_k(Y,\lambda)$ is defined as follows.

If $\lambda$ is a nondegenerate contact form on a three-manifold $Y$ diffeomorphic to $S^3$ whose kernel is the fillable contact structure, then $c_k(Y,\lambda)$ is the infimum over $L$ such that the degree $2k$ generator of $ECH(Y,\lambda,0)$ is contained in the image of the inclusion-induced map $ECH^L(Y,\lambda,0)\to ECH(Y,\lambda,0)$. Equivalently, if we pick a generic $\lambda$-compatible $J$, then $c_k(Y,\lambda)$ is the smallest $L$ such that the degree $2k$ generator of $ECH(Y,\lambda,0,J)$ can be represented in the chain complex $ECC(Y,\lambda,0,J)$ by a linear combination of admissible orbit sets each of which has action at most $L$. If $\lambda$ is possibly degenerate, then
\begin{equation}
\label{eqn:cklimit}
c_k(Y,\lambda) = \lim_{n\to\infty}c_k(Y,f_n\lambda)
\end{equation}
where $f_n:Y\to\R^{>0}$ are smooth functions such that $f_n\lambda$ is nondegenerate and $\lim_{n\to\infty}f_n=1$ in the $C^0$ topology.

\subsection{Cobordisms}
\label{sec:cob}

Let $(Y_+,\lambda_+)$ and $(Y_-,\lambda_-)$ be closed three-manifolds with contact forms. A {\em strong symplectic cobordism\/} ``from'' $(Y_+,\lambda_+)$ ``to'' $(Y_-,\lambda_-)$ is a compact symplectic four-manifold $(X,\omega)$ with boundary $\partial X = Y_+ - Y_-$ such that $\omega|_{Y_\pm}=d\lambda_\pm$.

Given $(X,\omega)$ as above, one can choose a neighborhood $N_-$ of $Y_-$ in $X$, identified with $[0,\epsilon)\times Y_-$ for some $\epsilon>0$, on which $\omega=e^s\lambda_-$, where $s$ denotes the $[0,\epsilon)$ coordinate. Likewise one can choose a neighborhood $N_+$ of $Y_+$ in $X$, identified with $(-\epsilon,0]\times Y_+$, on which $\omega=e^s\lambda_+$.
Given these choices, we define the ``completion'' of $(X,\omega)$ to be the four-manifold
\[
\overline{X} = ((-\infty,0]\times Y_-)\cup_{Y_-}X\cup_{Y_+} ([0,\infty)\times Y_+),
\]
glued using the above neighborhood identifications.

An almost complex structure $J$ on $\overline{X}$ is ``cobordism-admissible'' if it is $\omega$-compatible on $X$, and if on $[0,\infty)\times Y_+$ and $(-\infty,0]\times Y_-$ it agrees with $\lambda_\pm$-compatible almost complex structures $J_\pm$. Fix a cobordism-admissible $J$. If $\alpha_\pm$ are orbit sets in $Y_\pm$, we define the set $\mc{M}^J(\alpha_+,\alpha_-)$ of $J$-holomorphic currents in $\overline{X}$ analogously to the previous case of holomorphic currents in $\R\times Y$. The index inequality \eqref{eqn:indineq} and the topological complexity bound \eqref{eqn:tcb} carry over to $J$-holomorphic currents in the completed cobordism $\overline{X}$ without multiply covered components. The only difference is that in the definitions of $\op{ind}$, $I$, and $J_0$, now $c_\tau$ indicates the relative first Chern class of $T\overline{X}$.

In connection with cobordism maps on ECH, we will also need to consider ``broken holomorphic currents'':

\begin{definition}
Fix a cobordism-admissible almost complex structure $J$ on $\overline{X}$ which restricts to $\lambda_\pm$-compatible almost complex structures $J_\pm$ on the ends. Fix orbit sets $\alpha_\pm$ in $Y_\pm$. A {\em broken $J$-holomorphic current\/} from $\alpha_+$ to $\alpha_-$ is a tuple $\mc{B} = (\current(N_-),\current(N_-+1),\ldots,\current(N_+))$ where $N_-\le 0\le N_+$, for which there exist orbit sets $\alpha_-=\alpha_-(N_-),\ldots,\alpha_-(0)$ in $Y_-$, and orbit sets $\alpha_+(0),\ldots\alpha_+(N_+)=\alpha_+$ in $Y_+$,  such that:
\begin{itemize}
\item $\current_i\in\mc{M}^{J_-}(\alpha_-(i+1),\alpha_-(i))/\R$ for $i=N_-,\ldots,-1$.
\item $\current_0\in\mc{M}^J(\alpha_+(0),\alpha_-(0))$.
\item $\current_i\in\mc{M}^{J_+}(\alpha_+(i),\alpha_+(i-1))/\R$ for $i=1,\ldots,N_+$.
\item If $i\neq 0$, then not every component of $\current_i$ is a trivial cylinder.
\end{itemize}
\end{definition}

The holomorphic currents $\current(i)$ are called the ``levels'' of the broken holomorphic current $\mc{B}$. The ECH index of the broken $J$-holomorphic current $\mc{B}$ is defined to be the sum of the ECH indices of its levels:
\[
I(\mc{B}) = \sum_{i=N_-}^{N_+}I(\current(i)).
\]

\subsection{Cobordism maps}

We now consider maps on ECH induced by ``weakly exact'' symplectic cobordisms. For maps on ECH induced by exact symplectic cobordisms, see \cite[Thm.\ 1.9]{cc2}, and for general strong symplectic cobordisms see \cite{field}.

\begin{definition}
\label{def:we}
We call the strong symplectic cobordism $(X,\omega)$ {\em weakly exact\/} if there exists a $1$-form $\lambda$ on $X$ such that $d\lambda=\omega$.
\end{definition}

\begin{theorem}
\label{thm:sw}
Let $(Y_+,\lambda_+)$ and $(Y_-,\lambda_-)$ be closed oriented three-manifolds with contact forms, and let $(X,\omega)$ be a weakly exact symplectic cobordism from $(Y_+,\lambda_+)$ to $(Y_-,\lambda_-)$. Then there are canonical maps
\[
\Phi^L(X,\omega): ECH^L(Y_+,\lambda_+,0) \longrightarrow ECH^L(Y_-,\lambda_-,0)
\]
for each $L\in\R$ with the following properties:
\begin{description}
\item{(a)}
If $L<L'$, then the diagram
\[
\begin{CD}
ECH^L(Y_+,\lambda_+,0) @>{\Phi^L(X,\omega)}>> ECH^L(Y_-,\lambda_-,0)\\
@VVV  @VVV\\
ECH^{L'}(Y_+,\lambda_+,0) @>{\Phi^{L'}(X,\omega)}>> ECH^{L'}(Y_-,\lambda_-,0)
\end{CD}
\]
commutes. In particular,
\[
\Phi(X,\omega) = \lim_{L\to\infty}\Phi^L(X,\omega):ECH(Y_+,\lambda_+,0) \longrightarrow ECH(Y_-,\lambda_-,0)
\]
is well-defined.
\item{(b)}
If $X$ is diffeomorphic to a product $[0,1]\times Y$, then $\Phi(X,\omega)$ is an isomorphism.
\item{(c)}
If $J$ is any cobordism-admissible almost complex structure on $\overline{X}$, restricting to generic $\lambda_\pm$-compatible almost complex structures $J_\pm$ on the ends, then $\Phi^L(X,\omega)$ is induced by a (noncanonical) chain map
\[
\phi: ECC^L(Y_+,\lambda_+,0,J_+) \longrightarrow ECC^L(Y_-,\lambda_-,0,J_-)
\]
such that:
\begin{description}
\item{(i)} If $\alpha_\pm$ are admissible orbit sets for $\lambda_\pm$ with $[\alpha_\pm]=0$ and $\mc{A}(\alpha_\pm)<L$, and if the coefficient $\langle\phi\alpha_+,\alpha_-\rangle\neq 0$, then there exists a broken $J$-holomorphic current $\mc{B}$ from $\alpha_+$ to $\alpha_-$.
\item{(ii)} The broken $J$-holomorphic current $\mc{B}$ in (i) satisfies $I(\mc{B})=0$.
\end{description}
\end{description}
\end{theorem}

\begin{proof}
Assertions (a) and (b) are contained in \cite[Thm.\ 2.3]{qech}. Assertion (c)(i) is proved as in the ``Holomorphic Curves axiom'' in \cite[Thm.\ 1.9]{cc2}, modified for the weakly exact case as in \cite[Thm.\ 2.3]{qech}. Assertion (c)(ii) then follows from \cite[Thm.\ 5.1]{gradings}.
\end{proof}

\section{Multiply covered curves in cobordisms}
\label{sec:newECH}

The ECH cobordism maps in Theorem~\ref{thm:sw} are constructed using Seiberg-Witten theory. A technical difficulty with understanding ECH cobordism maps more directly in terms of holomorphic curves is that Proposition~\ref{prop:lowI} does not carry over to cobordisms. If $X$ is a strong symplectic cobordism, and if $J$ is a generic cobordism-admissible almost complex structure on $X$, then there may exist $J$-holomorphic currents $\current$ in $\overline{X}$ with multiply covered components such that $I(\current)<0$. Consequently, the broken holomorphic currents that arise in Theorem~\ref{thm:sw}(c) may be very complicated.

This section introduces a special kind of cobordism, called ``$L$-tame'', to which Proposition~\ref{prop:lowI} does carry over in a certain sense, stated in Proposition~\ref{prop:key} below.

\subsection{$L$-tame cobordisms}
\label{sec:deftame}

Let $(Y,\lambda)$ be a nondegenerate contact three-manifold.
If $\gamma$ is an elliptic Reeb orbit, then the linearized return map $P_\gamma$ is conjugate to rotation by angle $2\pi\theta$ for some $\theta\in\R/\Z$. Our assumption that $\lambda$ is nondegenerate implies that $\theta$ is irrational. We call $\theta$ the {\em rotation angle\/} of $\gamma$.

\begin{definition}
\label{def:lp}
Let $L>0$ and let $\gamma$ be an embedded elliptic Reeb orbit with action $\mc{A}(\gamma)<L$.
\begin{itemize}
\item
We say that $\gamma$ is {\em $L$-positive\/} if its rotation angle $\theta\in(0,\mc{A}(\gamma)/L)\mod 1$.
\item
We say that $\gamma$ is {\em $L$-negative\/} if its rotation angle $\theta\in(-\mc{A}(\gamma)/L,0)\mod 1$.
\end{itemize}
\end{definition}

Now let $(X,\omega)$ be a strong symplectic cobordism from $(Y_+,\lambda_+)$ to $(Y_-,\lambda_-)$ where $(Y_\pm,\lambda_\pm)$ are nondegenerate contact three-manifolds. 

\begin{definition}
If $\alpha_\pm$ are orbit sets for $\lambda_\pm$, define $e_L(\alpha_+,\alpha_-)$ to be the total multiplicity of all elliptic orbits in $\alpha_+$ that are $L$-negative, plus the total multiplicity of all elliptic orbits in $\alpha_-$ that are $L$-positive\footnote{That is, if $\alpha_+=\{(\alpha_i^+,m_i^+)\}$ and $\alpha_-=\{(\alpha_j^-,m_j^-)\}$, then we are to take the sum of $m_i^+$ over all $i$ such that $\alpha_i^+$ is elliptic and $L$-negative, plus the sum of $m_j^-$ over all $j$ such that $\alpha_j^-$ is elliptic and $L$-positive.}. If $\current\in\mc{M}^J(\alpha_+,\alpha_-)$, write $e_L(\current) = e_L(\alpha_+,\alpha_-)$.
\end{definition}

Let $J$ be a cobordism-admissible almost complex structure on $\overline{X}$, and let $L\in\R$. If $C$ is an irreducible $J$-holomorphic curve from $\alpha_+$ to $\alpha_-$, let $g(C)$ denote the genus of $C$, and let $h(C)$ denote the number of ends of $C$ at hyperbolic Reeb orbits.

\begin{definition}
\label{def:Ltame}
$(X,\omega,J)$ is {\em $L$-tame\/} if whenever $C$ is an embedded irreducible $J$-holomorphic curve in $\mc{M}^J(\alpha_+,\alpha_-)$, if there exists a positive integer $d$ such that $\mc{A}(\alpha_\pm)<L/d$ and $I(dC)\le 0$, then 
\begin{equation}
\label{eqn:Ltame}
2g(C)-2+\op{ind}(C) + h(C) + 2e_L(C)\ge 0.
\end{equation}
\end{definition}

\begin{example}
Suppose $(X,\omega)$ is a closed symplectic four-manifold (regarded as a cobordism from the empty set to itself). If $X$ contains no symplectically embedded sphere of self-intersection $-1$, then $(X,\omega,J)$ is $L$-tame for any $L$ and any generic $\omega$-compatible $J$.
\end{example}

\begin{example}
Suppose $(X,\omega)$ is a closed symplectic four-manifold and $T\subset X$ is an embedded Lagrangian torus such that $X\setminus T$ contains no symplectically embedded sphere of self-intersection $-1$. One can choose a tubular neighborhood $N$ of $T$ such that $X\setminus N$ is a strong symplectic cobordism from the empty set to the unit cotangent bundle of $T^2$. For any $L$, one can perturb $X\setminus N$ to $X'$ so that all embedded Reeb orbits on $\partial X'$ with symplectic action less than $L$ are positive hyperbolic or elliptic and $L$-positive. Then $(X',\omega,J)$ is $L$-tame for any generic cobordism-admissible $J$.
\end{example}

For the present paper, we will actually be interested in different examples of $L$-tame cobordisms which arise in \S\ref{sec:prove}.

The significance of the $L$-tameness condition is the following counterpart of Proposition~\ref{prop:lowI}:

\begin{proposition}
\label{prop:key}
Suppose that $J$ is generic and $(X,\omega,J)$ is $L$-tame. Let
\[
\current=\sum_kd_kC_k\in\mc{M}^J(\alpha_+,\alpha_-)
\]
be a $J$-holomorphic current in $\overline{X}$ with $\mc{A}(\alpha_\pm)<L$. Then:
\begin{description}
\item{(a)}
$I(\current)\ge 0$.
\item{(b)}
If $I(\current)=0$, then:
\begin{itemize}
\item
 $I(C_k)=0$ for each $k$.
\item
If $i\neq j$, then $C_i$ and $C_j$ do not both have positive ends at covers of the same $L$-negative elliptic orbit, and $C_i$ and $C_j$ do not both have negative ends at covers of the same $L$-positive elliptic orbit.
\item If $d_k'$ are integers with $0\le d_k' \le d_k$, then
\[
I\left(\sum_kd_k'C_k\right) = 0.
\]
\end{itemize}
\end{description}
\end{proposition}

\subsection{The ECH index of multiple covers and unions}
\label{sec:5.1}

The proof of Proposition~\ref{prop:key} uses a lower bound on the ECH index of multiply covered holomorphic curves and unions thereof in cobordisms, which we now state in Proposition~\ref{prop:5.1} below. Continue to fix a strong symplectic cobordism $(X,\omega)$ and a cobordism-admissible almost complex structure $J$ on $\overline{X}$.

\begin{definition}
\label{def:starL}
If $L>0$, and if $\current=\sum_id_iC_i\in\mc{M}^J(\alpha_+,\alpha_-)$ and $\current'=\sum_jd_j'C_j'\in\mc{M}^J(\alpha_+',\alpha_-')$ are $J$-holomorphic currents with $\mc{A}(\alpha_+\alpha_+')<L$ and $\mc{A}(\alpha_-\alpha_-') < L$, define an integer
\[
\current\star_L\current' = \sum_i\sum_jd_id_j'C_i\star_L C_j',
\]
where the integer $C_i\star_L C_j'$ is defined as follows:
\begin{itemize}
\item
If $C,C'$ are somewhere injective, irreducible, and distinct, then $C\star_L C'$ is the algebraic count of intersections of $C$ and $C'$. Note that there are only finitely many such intersections by \cite[Cor.\ 2.5]{siefring}. By intersection positivity, $C\star_L C'\ge 0$, with equality if and only if $C$ and $C'$ do not intersect.
\item
If $C$ is somewhere injective and irreducible, then
\begin{equation}
\label{eqn:CstarC}
C\star_L C = \frac{1}{2}\left(2g(C) - 2 + \op{ind}(C) + h(C) + 2e_L(C) + 4\delta(C)\right)
\end{equation}
where $\delta(C)$ is the count of singularities of $C$ with positive integer weights that appears in the relative adjunction formula, see \cite[\S3.3]{bn}. In particular, $\delta(C)\ge 0$, with equality if and only if $C$ is embedded.
\end{itemize}
\end{definition}

\begin{proposition}
\label{prop:5.1}
If $\current=\sum_id_iC_i$ and $\current'=\sum_jd_j'C_j'$ are $J$-holomorphic currents as in Definition~\ref{def:starL}, then:
\begin{description}
\item{(a)}
\[
I(\current + \current') \ge I(\current) + I(\current') + 2\current \star_L \current'.
\]
\item{(b)}
If equality holds in (a), and if $C_i\neq C_j'$, then $C_i$ and $C_j'$ do not both have positive ends at covers of the same $L$-negative elliptic orbit, and $C_i$ and $C_j'$ do not both have negative ends at covers of the same $L$-positive elliptic orbit.
\end{description}
\end{proposition}

\begin{remark}
Proposition~\ref{prop:5.1}(a) is a slight enhancement of \cite[Thm.\ 5.1]{ir}. The new element here is the $e_L(C)$ term in \eqref{eqn:CstarC}.
\end{remark}

\subsection{Proof of Proposition~\ref{prop:key}}

(a) For each $k$, since $C_k$ is somewhere injective and $J$ is generic, it follows that $\op{ind}(C_k)\ge 0$. By the index inequality \eqref{eqn:indineq}, we then have
\begin{equation}
\label{eqn:Ick}
I(C_k)\ge 0.
\end{equation}

Since $(X,\omega,J)$ is $L$-tame, if $I(d_kC_k)\le 0$ then
\begin{equation}
\label{eqn:gCk}
C_k \star_L C_k \ge 0.
\end{equation}
(When $C_k$ is embedded this follows from the definition of $L$-tame; and when $C_k$ is not embedded, $C_k \star_L C_k > 0$ because of the $\delta(C)$ term in \eqref{eqn:CstarC}.) Together with \eqref{eqn:Ick} and Proposition~\ref{prop:5.1}(a), this implies that $I(d_kC_k)\ge 0$. Thus
\begin{equation}
\label{eqn:regardless}
I(d_kC_k)\ge 0
\end{equation}
regardless.

By Proposition~\ref{prop:5.1}(a) again, we have
\begin{equation}
\label{eqn:5.1again}
I(\current) \ge \sum_kI(d_kC_k) + \sum_{i\neq j} d_id_j C_i\star_L C_j.
\end{equation}
Since $C_i\star_L C_j\ge 0$ for $i\neq j$, we conclude that $I(\current)\ge 0$.

(b) Suppose that $I(\current)=0$. Then by \eqref{eqn:regardless} and \eqref{eqn:5.1again} we have
\begin{equation}
\label{eqn:IdC0}
I(d_kC_k)=0
\end{equation}
for each $k$. Also, since equality holds in \eqref{eqn:5.1again}, it follows from Proposition~\ref{prop:5.1}(b) that the second bullet in Proposition~\ref{prop:key}(b) holds.

To prove the first bullet in Proposition~\ref{prop:key}(b), note that for each $k$, it follows from equation \eqref{eqn:IdC0} and the definition of $L$-tame that the inequality \eqref{eqn:gCk} holds.  This, together with equation \eqref{eqn:IdC0} and Proposition~\ref{prop:5.1}(a), implies that $I(C_k)\le 0$. Hence by \eqref{eqn:Ick} we have $I(C_k)=0$.

The third bullet in Proposition~\ref{prop:key}(b) holds because, since we have seen above that \eqref{eqn:gCk} holds for each $k$, it follows from Proposition~\ref{prop:5.1}(a) that
\[
0=I(\current) \ge I\left(\sum_kd_k'C_k\right) + I\left(\sum_k(d_k-d_k')C_k\right),
\]
and both of the sums on the right are nonnegative.

\subsection{Proof of Proposition~\ref{prop:5.1}}

We follow the proof of \cite[Thm.\ 5.1]{ir} with minor modifications.

Let $\{C_a\}$ denote the union of the sets of components of $\current$ and $\current'$.
Let $d_a$ denote the multiplicity of $C_a$ in $\current$, and let $d_a'$ denote the multiplicity of $C_a$ in $\current'$.

Let $\gamma$ be an embedded elliptic Reeb orbit such that some $C_a$ has a positive end at a cover of $\gamma$. To prove (a), as in \cite[Eq.\ (5.6)]{ir}, it is enough to prove the following inequality (one also needs an analogous inequality for the negative ends which follows by symmetry):
\begin{gather}
\label{eqn:5.6}
\left(\sum_{k=1}^{M+M'} - \sum_{k=1}^M - \sum_{k=1}^{M'}\right) CZ_\tau(\gamma^k)
\quad\quad\quad\quad\quad\quad\quad\quad\quad\quad
\quad\quad\quad\quad\quad\quad\quad\quad\quad\quad
\\
\nonumber
\quad\quad\quad
\ge 2\sum_{a\neq b}d_ad_b'\ell_\tau(\zeta_a,\zeta_b) + \sum_a d_ad_a'\left(-\epsilon n_a + 2\epsilon_Lm_a + \sum_{i=1}^{n_a}CZ_\tau(\gamma^{q_{a,i}}) + 2w_\tau(\zeta_a)\right).
\end{gather}
Here the notation is as follows. For each $a$, let $n_a$ denote the number of positive ends of $C_a$ at covers of $\gamma$, and denote their covering multiplicities by $q_{a,1},\ldots,q_{a,n_a}$. Let $m_a=\sum_{i=1}^{n_a}q_{a,i}$ denote the total covering multiplicity of the positive ends of $C_a$ at covers of $\gamma$. Let $M=\sum_ad_am_a$ and $M'=\sum_ad_a'm_a$. Let $\zeta_a$ denote the braid in a neighborhood of $\gamma$ determined by the asymptotics of the positive ends of $C_a$, see \cite[\S3.3]{bn}. Let $w_\tau(\zeta_a)$ denote writhe of this braid with respect to a trivialization $\tau$ of $\xi|_\gamma$, see \cite[\S3.3]{bn}. For $a\neq b$, let $\ell_\tau(\zeta_a,\zeta_b)$ denote the linking number of the braids $\zeta_a$, $\zeta_b$ with respect to $\tau$, see \cite[\S5.1]{bn}. Let $\epsilon$ denote $1$ if $\gamma$ is elliptic and $0$ otherwise. Finally, let $\epsilon_L$ denote $1$ if $\gamma$ is elliptic and $L$-negative, and $0$ otherwise.

The inequality \eqref{eqn:5.6} is the same as \cite[Eq.\ (5.6)]{ir} except for the new $\epsilon_L$ term. In particular, if $\gamma$ is not both elliptic and $L$-negative, then \eqref{eqn:5.6} is proved in \cite{ir}. So we just need to prove \eqref{eqn:5.6} when $\gamma$ is elliptic and $L$-negative.

In this case, we can choose the trivialization $\tau$ so that $CZ_\tau(\gamma^k)=-1$ for all $k<L/\mc{A}(\gamma)$, see \cite[Eq.\ (3.2)]{bn}. It follows that the left hand side of \eqref{eqn:5.6} equals zero.
The linking bound in \cite[Lem.\ 5.5(b)]{bn} implies that if $a\neq b$ then
\[
\ell_\tau(\zeta_a,\zeta_b) \le -\sum_{i=1}^{n_a}\sum_{j=1}^{n_b} \min(q_{a,i},q_{b,j}).
\]
In particular, $\ell_\tau(\zeta_a,\zeta_b)<0$ whenever $n_an_b\neq 0$. Thus the first sum on the right hand side of \eqref{eqn:5.6} is less than or equal to zero, with equality only if $d_ad_b'=0$ whenever $a\neq b$ and $n_an_b\neq 0$. Finally, the writhe bound in \cite[Eq. (3.9)]{bn} implies that 
\[
w_\tau(\zeta_a) \le -m_a+n_a
\]
(with equality only if $n_a=1$). Since $CZ_\tau(\gamma^{q_{a,i}})=-1$ for each $i$, it follows that the second sum on the right hand side of \eqref{eqn:5.6} is less than or equal to zero. This completes the proof of \eqref{eqn:5.6} and hence of part (a).

The above paragraph also shows that if $\gamma$ is elliptic and $L$-negative, then equality holds in \eqref{eqn:5.6} only if
$d_ad_b'=0$ whenever $a\neq b$ and $n_an_b\neq 0$. This (together with its analogue for negative ends) implies (b).

\section{The boundary of a convex toric domain}
\label{sec:perturb}

With the above generalities about ECH out of the way, we now return to the specific situation in Theorem~\ref{thm:convex}. In this section
we study the ECH chain complex of the boundary of a suitably perturbed convex toric domain. The information we need for Theorem~\ref{thm:convex} is extracted in Lemmas~\ref{lem:perturbation} and \ref{lem:supplement} below.

\subsection{The perturbation}
\label{sec:perturbation}

If $\nu$ is an edge of a convex integral path $\Lambda$, let $m(\nu)\in\{1,\ldots\}$ denote the multiplicity of $\nu$, namely one plus the number of lattice points in the interior of $\nu$.

\begin{definition}
An {\em extended convex generator\/} is a convex integral path $\Lambda$ such that:
\begin{itemize}
\item Each edge $\nu$ of $\Lambda$ is labeled by an integer $l(\nu)\in\{0,\ldots,m(\nu)\}$.
\item Horizontal and vertical edges are labeled by $0$.
\end{itemize}
\end{definition}

Note that a convex generator determines an extended convex generator, in which each `$e$' label is replaced by $0$, and each `$h$' label is replaced by $1$. Like a convex generator, an extended convex generator can be represented by a commutative formal product of symbols $e_{a,b}$ and $h_{a,b}$. Now, however, the exponent of $h_{a,b}$ can be greater than one. If $a$ and $b$ are relatively prime nonnegative integers, then an edge with displacement vector $(ma,-mb)$ labeled by the integer $l$ corresponds to the formal product $e_{a,b}^{m-l}h_{a,b}^l$.

If $\Lambda$ is an extended convex generator, define its ECH index $I(\Lambda)$ by \eqref{eqn:ILambda}, where now $h(\Lambda)$ is defined as follows:

\begin{definition}
If $\Lambda$ is an extended convex generator, define $h(\Lambda)$ to be the sum of the integer labels of all of its edges. Define $e(\Lambda)$ to be the number of edges $\nu$ such that $l(\nu)<m(\nu)$. In terms of formal products, $h(\Lambda)$ is the total exponent of all $h_{a,b}$ factors, and $e(\Lambda)$ is the number of distinct factors $e_{a,b}$ that appear in $\Lambda$.
\end{definition}

If $X_\Omega$ is a convex toric domain and $\Lambda$ is an extended convex generator, define the action $A_\Omega(\Lambda)$ by \eqref{eqn:ALambda} as before.

\begin{definition}
\label{def:smooth}
A convex toric domain $X_\Omega$ is {\em smooth\/} if the function $f$ defining $\Omega$ in \eqref{eqn:Omega} is smooth, $f(A)=0$, and $f'$ is constant near $0$ and $A$.
\end{definition}

If $X_\Omega$ is a smooth convex toric domain, then $Y=\partial X_\Omega$ is a smooth star-shaped hypersurface in $\R^4$, so $Y$ is diffeomorphic to $S^3$ and the 1-form $\lambda_{std}$ in \eqref{eqn:lambdastd} restricts to a contact form on $Y$. The following lemma will be proved in \S\ref{sec:ppl}:

\begin{lemma}
\label{lem:perturbation}
Let $X_\Omega$ be a smooth convex toric domain with boundary $Y$. Then for every $\epsilon,L>0$, there is a contact form $\lambda$ on $Y$ with the following properties:
\begin{description}
\item{(a)}
$\lambda$ is nondegenerate.
\item{(b)}
$\lambda = f\left(\lambda_{std}|_Y\right)$ where $f:Y\to \R^{>0}$ is a smooth function with $\|f-1\|_{C^0}<\epsilon$.
\item{(c)}
All hyperbolic orbits with action less than $L$ are positive hyperbolic, and all embedded elliptic orbits with action less than $L$ are $L$-positive, see Definition~\ref{def:lp}.
\item{(d)}
There is a bijection $\imath$, from the set of extended convex generators $\Lambda$ with $A_\Omega(\Lambda)<L$ to the set of orbit sets $\alpha$ for $\lambda$ with $\mc{A}(\alpha) < L$, such that if $\alpha=\imath(\Lambda)$, then:
\begin{description}
\item{(i)}
$\alpha$ is admissible if and only if $\Lambda$ is a convex generator.
\item{(ii)}
$|\mc{A}(\alpha) - A_\Omega(\Lambda)| < \epsilon$.
\item{(iii)}
$I(\alpha) = I(\Lambda)$.
\item{(iv)}
$J_0(\alpha) = I(\Lambda) - 2x(\Lambda) - 2y(\Lambda) - e(\Lambda)$.
\end{description}
\end{description}
\end{lemma}

Note that in part (d) above, $I(\alpha)\in\Z$ is defined by \eqref{eqn:absoluteZgrading}, and similarly $J_0(\alpha)\in\Z$ is defined by 
\begin{equation}
\label{eqn:absJ0}
J_0(\alpha)=J_0(\alpha,\emptyset,Z),
\end{equation}
where $Z$ is the unique class in $H_2(Y,\alpha,\emptyset)$.

\subsection{ECH capacities of convex toric domains}

The proof of Theorem~\ref{thm:convex} will need the following consequence of Lemma~\ref{lem:perturbation}:

\begin{lemma}
\label{lem:supplement}
Let $X_\Omega$ be a smooth convex toric domain, and let $\Lambda$ be a convex generator which is minimal for $X_\Omega$. If $L>A_\Omega(\Lambda)$, if $\epsilon>0$ is sufficiently small, and if $\lambda$ is a contact form on $Y=\partial X_\Omega$ provided by Lemma~\ref{lem:perturbation}, then for any generic $\lambda$-compatible almost complex structure $J$ on $\R\times Y$, the admissible orbit set $\imath(\Lambda)$ is a cycle in $ECC^L(Y,\lambda,0,J)$ which represents a nontrivial homology class in $ECH(Y,\lambda,0)$.
\end{lemma}

To prove Lemma~\ref{lem:supplement}, we will use the following formula for the ECH capacities of convex toric domains.

\begin{proposition}
\label{prop:dan}
Let $X_\Omega$ be a convex toric domain. Then
\begin{equation}
\label{eqn:dan}
c_k(X_\Omega) = \min\{A_{\Omega}(\Lambda)\mid L(\Lambda)= k+1\}
\end{equation}
where the minimum is over convex integral paths $\Lambda$.
\end{proposition}

\begin{proof}
A special case of \cite[Cor.\ A.5]{dcg} (which applies to some more general domains) asserts that \eqref{eqn:dan} holds where the minimum is taken over a larger set of paths, let us call these ``generalized convex integral paths'',  which are defined as in Definition~\ref{def:cip}, except that the angle of the tangent vector can range over the interval $(\pi/2,-\pi)$, instead of just $[0,-\pi/2]$. Consequently, it is enough to show that for our domain $X_\Omega$, the minimum over generalized convex integral paths is the same as the minimum over convex integral paths. For any generalized convex integral path, one can replace all initial edges with positive angle by a single horizontal edge with the same total horizontal displacement (which adds some lattice points to the enclosed region); and one can likewise replace all final edges with angle less than $-\pi/2$ by a single vertical edge. This does not change $A_\Omega$. One can then ``round corners'' as in \cite[\S1.3.3]{t3} to reduce the number of lattice points to what it was before; this decreases $A_\Omega$ similarly to \cite[Lem.\ 2.14]{t3}. Thus every generalized convex integral path $\Lambda$ can be replaced by a convex integral path $\Lambda'$ with $L(\Lambda')=L(\Lambda)$ and $A_\Omega(\Lambda')\le A_\Omega(\Lambda)$.

Another way to prove \eqref{eqn:dan}, pointed out to the author by K.\ Choi and V.\ Ramos, is as follows. Let $\Omega'$ be a translate of $\Omega$ in the first quadrant which does not intersect the axes. As explained in \cite{lmt}, a version of the Traynor trick from \cite{traynor} can be used to show that for any $\epsilon>0$, there exist symplectic embeddings
\[
X_{(1+\epsilon)^{-1}\Omega} \to X_{\Omega'} \to X_{(1+\epsilon)\Omega}.
\]
It then follows from the Monotonicity and Conformality axioms of ECH capacities, see \cite[\S1.2]{bn}, that
\[
c_k(X_\Omega) = c_k(X_{\Omega'})
\]
for all $k$. The ECH capacities of $X_{\Omega'}$ are computed by \cite[Thm.\ 4.14]{bn}. One can then deduce \eqref{eqn:dan} from the latter theorem similarly to the previous paragraph.
\end{proof}

We will also need the following fact about minimizers in the above formula:

\begin{lemma}
\label{lem:dan'}
Suppose that $\Lambda$ is a convex generator such that $I(\Lambda)=2k$ and $\Lambda$ is minimal for $X_\Omega$. Then $\Lambda$ uniquely minimizes $A_\Omega$ among all convex generators with $I=2k$.
\end{lemma}

\begin{proof}
It is enough to show that if $\Lambda$ is a convex generator with $I(\Lambda)=2k$ such that not all edges of $\Lambda$ are labeled `$e$', then there exists a convex generator $\Lambda'$ with $I(\Lambda')=2k$ and $A_\Omega(\Lambda')<A_\Omega(\Lambda)$. Since $I(\Lambda)$ is even, it follows from the index formula \eqref{eqn:ILambda} that at least two edges of $\Lambda$ are labeled `$h$'. One can then ``round a corner'' of $\Lambda$ as in \cite[\S1.3.3]{t3} to reduce $L(\Lambda)$ by $1$, and one can also reduce the number of `$h$' labels by $2$. This will preserve the ECH index $I$ by \eqref{eqn:ILambda}, and decrease the symplectic action $A_\Omega$ similarly to \cite[Lem.\ 2.14]{t3}.
\end{proof}

\begin{proof}[Proof of Lemma~\ref{lem:supplement}.]
Write $I(\Lambda)=2k$. By Lemma~\ref{lem:dan'}, we can choose $\delta>0$ such that every convex generator $\Lambda'\neq\Lambda$ with $I(\Lambda')=2k$ has action
\[
A_\Omega(\Lambda') > A_\Omega(\Lambda) + \delta.
\]
We can assume that $\epsilon<\delta/2$, so that by condition (ii) in Lemma~\ref{lem:perturbation}(d), if $\Lambda'$ is any convex generator with $A_\Omega(\Lambda')<L$ then
\[
|\mc{A}(\imath(\Lambda')) - A_\Omega(\Lambda')| < \delta/2.
\]
If $\Lambda'\neq\Lambda$ is any convex generator with $A_\Omega(\Lambda')<L$ and $I(\Lambda')=2k$, then it follows from the above that
\begin{equation}
\label{eqn:adelta1}
\mc{A}(\imath(\Lambda')) > A_\Omega(\Lambda) + \delta/2.
\end{equation}

Now by Lemma~\ref{lem:perturbation}(b) and \eqref{eqn:cklimit}, if $\epsilon$ is sufficiently small then
\[
|c_k(Y,\lambda) - c_k(X_\Omega)| < \delta/2.
\]
Since $A_\Omega(\Lambda) = c_k(X_\Omega)$ by Proposition~\ref{prop:dan}, this implies that
\begin{equation}
\label{eqn:adelta2}
c_k(Y,\lambda) < A_\Omega(\Lambda) + \delta/2.
\end{equation}

By the definition of $c_k$, there is a cycle in $ECC^L(Y,\lambda,0,J)$ which is a linear combination of admissible orbit sets $\alpha$ with $I(\alpha)=2k$ and $\mc{A}(\alpha)\le c_k(Y,\lambda)$, and which represents a nontrivial homology class in $ECH(Y,\lambda,0)$. But by \eqref{eqn:adelta1} and \eqref{eqn:adelta2}, $\alpha=\imath(\Lambda)$ is the only admissible orbit set for $\lambda$ such that $I(\alpha) = 2k$ and $\mc{A}(\alpha) \le c_k(Y,\lambda)$. Hence this cycle must be $\imath(\Lambda)$.
\end{proof}

\subsection{Proof of the perturbation lemma}
\label{sec:ppl}

\begin{proof}[Proof of Lemma~\ref{lem:perturbation}.]
This is similar to the proof of \cite[Lem.\ 3.3]{concave}. We proceed in five steps.

{\em Step 1.\/}
For reasons to be explained in Step 5, we can assume without loss of generality that the function $f:[0,A]\to\R^{>0}$ defining $\Omega$ in \eqref{eqn:Omega} has the following three properties: First, $f'(0)$ is irrational and
\begin{equation}
\label{eqn:fprime0}
|f'(0)|<f(0)/L.
\end{equation}
Second, $f'(A)$ is irrational and
\begin{equation}
\label{eqn:fprimeA}
|f'(A)|>L/A.
\end{equation}
Third, $f''(t)<0$ except for $t$ in connected neighborhoods of $0$ and $A$.

{\em Step 2.\/}
We now list the embedded Reeb orbits of $\lambda_{std}|_Y$. Similarly to \cite[\S4.3]{bn}, these are described as follows. Let $\mu:\C^2\to\R^2$ denote the moment map $\mu(z)=(\pi|z_1|^2,\pi|z_2|^2)$.
\begin{itemize}
\item
The circle
\[
e_{1,0}=\mu^{-1}(0,f(0))
\]
is an elliptic Reeb orbit with action
\begin{equation}
\label{eqn:Ae10}
\mc{A}(e_{1,0})=f(0)
\end{equation}
and rotation angle
\begin{equation}
\label{eqn:theta0}
\theta = |f'(0)|\mod 1.
\end{equation}
Note that \eqref{eqn:fprime0}, \eqref{eqn:Ae10}, and \eqref{eqn:theta0} imply that $e_{1,0}$ is $L$-positive.
\item
The circle
\[
e_{0,1}=\mu^{-1}(A,0)
\]
is an elliptic Reeb orbit with action
\begin{equation}
\label{eqn:Ae01}
\mc{A}(e_{0,1})=A
\end{equation}
and rotation angle
\begin{equation}
\label{eqn:thetaA}
\theta = |1/f'(A)| \mod 1.
\end{equation}
Note that \eqref{eqn:fprimeA}, \eqref{eqn:Ae01}, and \eqref{eqn:thetaA} imply that $e_{0,1}$ is $L$-positive.
\item
For each $t\in(0,A)$ such that $f'(t)$ is rational, if we write $f'(t)=-b/a$ where $a,b$ are relatively prime positive integers, then the torus
\[
T_{a,b} = \mu^{-1}(t,f(t))
\]
is foliated by an $S^1$-family of embedded Reeb orbits, and each Reeb orbit $\gamma$ in this family has action
\[
\mc{A}(\gamma) = af(t)+bt.
\]
We denote the above quantity by $\mc{A}(T_{a,b})$.
\end{itemize}

{\em Step 3.\/}
Given $\delta>0$, we can perturb $\lambda_{std}|_Y$ to a nondegenerate contact form $\lambda=f(\lambda_{std}|_Y)$ with $\|f-1\|_{C^0}<\delta$ such that:
\begin{itemize}
\item $\lambda$ agrees with $\lambda_{std}|_Y$ near the Reeb orbits $e_{1,0}$ and $e_{0,1}$.
\item Each circle of Reeb orbits $T_{a,b}$ with $\mc{A}(T_{a,b})<L$ is replaced by an $L$-positive elliptic orbit $e_{a,b}$ and a positive hyperbolic orbit $h_{a,b}$.
\item
The orbits $e_{a,b}$ and $h_{a,b}$ both have action\footnote{The action of $e_{a,b}$ is necessarily greater than that of $h_{a,b}$, because for a suitable $\lambda$-admissible $J$ there are two $J$-holomorphic cylinders from $e_{a,b}$ to $h_{a,b}$.} less than $L$ and within $\delta$ of $\mc{A}(T_{a,b})$.
\item
$\lambda$ has no other embedded Reeb orbits of action less than $L$.
\end{itemize}
This perturbation is just like the one in the proof of \cite[Lem.\ 3.3]{concave} for concave toric domains, except that in \cite{concave} the elliptic orbits $e_{a,b}$ are $L$-negative instead of $L$-positive.

To clarify why our elliptic orbit $e_{a,b}$ arising from $T_{a,b}$ is $L$-positive, note that since $f''(t)<0$ away from $0$ and $A$, the linearized return map of each orbit in $T_{a,b}$ has the form $\begin{pmatrix}1 & 0 \\ 1 & 1\end{pmatrix}$ in some symplectic basis. The linearized return map of $e_{a,b}$ is a perturbation of this matrix, so the rotation angle of $e_{a,b}$ is positive, and can be made arbitrarily small by choosing the perturbation sufficiently small. In particular, if this rotation angle is less than $L/(\mc{A}(T_{a,b})-\delta)$, then $e_{a,b}$ will be $L$-positive.

The map $\imath$ from extended convex generators to orbit sets is now defined in the obvious way suggested by the notation. Namely, if $\Lambda$ is an extended convex generator, represented as a formal product of symbols $e_{a,b}$ and $h_{a,b}$ raised to various exponents, then $\imath(\Lambda)$ is the same formal product, now regarded as representing an orbit set involving the Reeb orbits $e_{a,b}$ and $h_{a,b}$. If $\delta$ is chosen sufficiently small with respect to $\epsilon$, then $\imath$ is a well defined bijection as in (d) satisfying conditions (i) and (ii).

{\em Step 4.\/} Under the assumptions in Step 1, the contact form $\lambda$ satisfies conditions (a), (b), and (c), and we have a bijection $\imath$ as in (d) satisfying conditions (i) and (ii). We now prove conditions (iii) and (iv).

We have
\[
I(\alpha) = c_\tau(\alpha) + Q_\tau(\alpha) + CZ_\tau^I(\alpha)
\]
where $c_\tau(\alpha)$ is short for $c_\tau(\alpha,\emptyset,Z)$ and $Q_\tau(\alpha)$ is short for $Q_\tau(\alpha,\emptyset,Z)$, where $Z$ is the unique class in $H_2(Y,\alpha,\emptyset)$. We can choose the trivialization $\tau$ so that $CZ_\tau(h_{a,b})=0$ and $CZ_\tau(e_{a,b})=1$. We saw in Steps 2 and 3 that all embedded elliptic orbits $e_{a,b}$ with action less than $L$ are $L$-positive. It follows that if $k$ is a positive integer such that $\mc{A}(e_{a,b}^k)<L$, then $CZ_\tau(e_{a,b}^k)=1$. Thus
\[
CZ_\tau^I(\alpha) = m(\Lambda) - h(\Lambda)
\]
where $m(\Lambda)$ denotes the total multiplicity of all edges of $\Lambda$.
Similarly to \cite[Eq.\ (3.14)]{concave}, we have
\begin{equation}
\label{eqn:cta}
c_\tau(\alpha) = x(\Lambda) + y(\Lambda).
\end{equation}
We also have
\[
Q_\tau(\alpha) = 2\op{Area}(R)
\]
where $R$ denotes the region bounded by $\Lambda$ and the axes. This is proved similarly\footnote{The only difference is that in our situation, in the equation above \cite[Eq.\ (3.14)]{concave}, `min' needs to be replaced by `max'.} to the equation below \cite[Eq.\ (3.14)]{concave}. Finally, Pick's formula for the area of a lattice polygon implies that
\[
2\op{Area}(R) = 2L(\Lambda) - m(\Lambda) - x(\Lambda) - y(\Lambda) - 2
\]
where $L(\Lambda)$ denotes the number of lattice points in $R$, including on the boundary. Combining the above five equations gives
\[
I(\alpha) = 2(L(\Lambda)-1)- h(\Lambda),
\]
which proves (iii).

By the definition of $J_0$ in \eqref{eqn:defJ0} and \eqref{eqn:absJ0}, we have
\[
J_0(\alpha) = I(\alpha) - 2c_\tau(\alpha) - e(\alpha).
\]
Combining this with (iii) and equation \eqref{eqn:cta} proves (iv).

{\em Step 5.\/} We now explain why we can make the assumptions in Step 1 without loss of generality.

Let $X_{\Omega}$ be any smooth convex toric domain, with boundary $Y$ and defining function $f$. For any $\delta>0$ there is a smooth convex toric domain $X_{\Omega_0}$ satisfying the assumptions in Step 1 with defining function $f_0$ such that $\|f-f_0\|_{C^0}<\delta$. The flow of the Liouville vector field \eqref{eqn:radialvf} defines a diffeomorphism $Y\simeq Y_0=\partial X_{\Omega_0}$ with respect to which
\[
\|\lambda_{std}|_Y - \lambda_{std}|_{Y_0}\|_{C^0}<\epsilon/2
\]
if $\delta$ is sufficiently small. In addition, if $\Lambda$ is a convex generator with $A_\Omega(\Lambda)<L$, then
\[
|A_\Omega(\Lambda) - A_{\Omega_0}(\Lambda)|<\epsilon/2
\]
if $\delta$ is sufficiently small, cf.\ \cite[Lem.\ 2.4]{concave}. Thus, the lemma for $\Omega$ follows from the lemma for $\Omega_0$ with $\epsilon$ replaced by $\epsilon/2$.
\end{proof}

\section{Proof of the main theorem}
\label{sec:prove}

\begin{proof}[Proof of Theorem~\ref{thm:convex}.] We proceed in five steps.

{\em Step 1:\/} We begin with some geometric setup.

We can assume, by slightly enlarging $X_{\Omega'}$ and shrinking $X_\Omega$, that $X_\Omega$ and $X_{\Omega'}$ are smooth as in Definition~\ref{def:smooth} and that there is a symplectic embedding $\varphi:X_\Omega\to\op{int}(X_{\Omega'})$. (The theorem for the original symplectic embedding then follows by a limiting argument.)

Let $Y$ and $Y'$ denote the boundaries of $X_\Omega$ and $X_{\Omega'}$ respectively. Then $X_{\Omega'}\setminus\varphi(\op{int}(X_\Omega))$ is a weakly exact symplectic cobordism from $(Y',\lambda_{std}|_{Y'})$ to $(Y,\lambda_{std}|_Y)$, see Definition~\ref{def:we}.

Choose $L>A_{\Omega'}(\Lambda')$ and $\epsilon>0$. Let $\lambda$ and $\lambda'$ be contact forms on $Y$ and $Y'$ respectively provided by Lemma~\ref{lem:perturbation}. If $\epsilon$ is sufficiently small, then there is also a weakly exact symplectic cobordism $(X,\omega)$, with $X$ diffeomorphic to $[0,1]\times S^3$, from $(Y',\lambda')$ to $(Y,\lambda)$, because multiplying the contact form $\lambda_{std}|_Y$ by a function $f$ on $Y$ is equivalent to moving $Y$ by the time $\log(f)$ flow of the Liouville vector field \eqref{eqn:radialvf}.

Let $J_+$ be a generic $\lambda'$-compatible almost complex structure on $\R\times Y'$, let $J_-$ be a generic $\lambda$-compatible almost complex structure on $\R\times Y$, and let $J$ be a generic cobordism-admissible almost complex structure on $\overline{X}$ restricting to $J_+$ and $J_-$ on the ends.

{\em Step 2:\/} We now show that the cobordism $(X,\omega,J)$ is $L$-tame, see Definition~\ref{def:Ltame}.

Let $d$ be a positive integer, let $\alpha_+$ be an orbit set for $\lambda'$ with $\mc{A}(\alpha_+)<L/d$, let $\alpha_-$ be an orbit set for $\lambda$ with $\mc{A}(\alpha_-)<L/d$, and let $C$ be an irreducible embedded curve in $\mc{M}^J(\alpha_+,\alpha_-)$ with $I(dC)\le 0$. We need to prove the inequality \eqref{eqn:Ltame}.

Suppose to get a contradiction that
\begin{equation}
\label{eqn:notLtame}
2g(C) + \op{ind}(C) + h(C) + 2e_L(C) \le 1.
\end{equation}
Since $J$ is generic, $\op{ind}(C)\ge 0$. By \eqref{eqn:ind}, the parity of $\op{ind}(C)$ equals the number of ends of $C$ at positive hyperbolic orbits, which equals $h(C)$ since $\alpha_+$ and $\alpha_-$ contain no negative hyperbolic orbits by Lemma~\ref{lem:perturbation}(c). Thus the inequality \eqref{eqn:notLtame} forces
\begin{equation}
\label{eqn:forces}
g(C) = \op{ind}(C) = h(C) = e_L(C) = 0.
\end{equation}

By Lemma~\ref{lem:perturbation}(c), every embedded elliptic orbit in $\alpha_-$ is $L$-positive, so by \eqref{eqn:forces} there are no elliptic orbits in $\alpha_-$, and thus $\alpha_-$ is the empty set. It follows that $I(dC)=I(\alpha_+^d)$. If $\Lambda_+$ is the extended convex generator with $\imath(\Lambda_+)=\alpha_+^d$, then by Lemma~\ref{lem:perturbation}(d), $I(\alpha_+^d)=I(\Lambda_+)$. But it follows from the definition of the ECH index of an extended convex generator in \S\ref{sec:perturbation} that $I(\Lambda_+)\ge 0$, with equality if and only if $\alpha_+^d$ is the empty set. But $\alpha_+^d$ is not the empty set since $C$ is nonempty. Thus $I(dC)> 0$, contradicting our hypothesis.

{\em Step 3:\/} We now show that there exists a $J$-holomorphic current
\[
\current\in\mc{M}^J(\imath(\Lambda'),\imath(\Lambda))
\]
for some convex generator $\Lambda$ with $I(\Lambda)=I(\Lambda')$ and $A_\Omega(\Lambda')<L$.

Let
\[
\phi: ECC^L(Y',\lambda',0,J_+) \longrightarrow ECC^L(Y,\lambda,0,J_-)
\]
be a chain map provided by Theorem~\ref{thm:sw}(c). By Lemma~\ref{lem:supplement}, we can choose $\epsilon$ sufficiently small so that $\alpha'=\imath(\Lambda')$ is a cycle in $ECC^L(Y',\lambda',0,J_+)$ which represents a nontrivial homology class in $ECH(Y',\lambda',0)$. Then $\phi(\alpha')\neq 0$ by Theorem~\ref{thm:sw}(b). Hence by Theorem~\ref{thm:sw}(c)(i), there is an admissible orbit set $\alpha$ for $\lambda$ with $\mc{A}(\alpha) < L$, and a broken holomorphic current
\[
\mc{B} = (\current({N_-}),\ldots,\current({N_+}))
\]
from $\alpha'$ to $\alpha$. By Theorem~\ref{thm:sw}(c)(ii), we have $I(\mc{B})=0$. Equivalently, if $\Lambda$ is the convex generator with $\alpha=\imath(\Lambda)$, then $I(\Lambda)=I(\Lambda')$.

By Proposition~\ref{prop:lowI}, we know that $I(\current(i))>0$ if $i\neq 0$. By Step 2 and Proposition~\ref{prop:key}(a), we have $I(\current(0))\ge 0$. Since $\sum_iI(\current(i))=0$, it follows that $N_-=N_+=0$ and $I(\current(0))=0$. Thus $\current=\current(0)$ is the desired $J$-holomorphic current.

{\em Step 4:\/} Write $\current=\sum_k d_kC_k$. Let $\Lambda_k'$ and $\Lambda_k$ denote the convex generators for which $C_k\in\mc{M}^J(\imath(\Lambda_k'),\imath(\Lambda_k))$. We now show that $\Lambda_k\le_{\Omega,\Omega'}\Lambda_k'$ ``up to $\epsilon$'' for each $k$, namely
\begin{align}
\label{eqn:lt1}
I(\Lambda_k) &=I(\Lambda_k'),\\
\label{eqn:lt2}
A_{\Omega}(\Lambda_k) &\le A_{\Omega'}(\Lambda_k') + 2\epsilon,\\
\label{eqn:lt3}
x(\Lambda_k) + y(\Lambda_k) - h(\Lambda_k)/2 &\ge x(\Lambda_k') + y(\Lambda_k') + m(\Lambda_k') - 1.
\end{align}

Equation \eqref{eqn:lt1} follows from the first bullet of Proposition~\ref{prop:key}(b) and assertion (iii) in Lemma~\ref{lem:perturbation}(d).

To prove \eqref{eqn:lt2}, let $\alpha_\pm(k)$ denote the admissible orbit sets such that $\imath(\Lambda_k')=\alpha_+(k)$ and $\imath(\Lambda_k)=\alpha_-(k)$. The existence of the curve $C_k\in\mc{M}^J(\alpha_+(k),\alpha_-(k))$, and the weak exactness of the cobordism, imply that $\mc{A}(\alpha_+(k)) \ge \mc{A}(\alpha_-(k))$, by a Stokes' theorem calculation in the proof of \cite[Thm.\ 2.3]{qech}. Equation \eqref{eqn:lt2} then follows from Lemma~\ref{lem:perturbation}(d)(ii).

To prove \eqref{eqn:lt3}, we apply Proposition~\ref{prop:J0} to $C_k$. By the definition of $J_0$, we have
\[
J_0(C_k) = J_0(\alpha_+(k)) - J_0(\alpha_-(k)),
\]
see \cite[Prop.\ 6.5(a)]{ir}. Then by Lemma~\ref{lem:perturbation}(d)(iv) and equation \eqref{eqn:lt1}, we have
\begin{equation}
\label{eqn:J0comb}
J_0(C_k) = -2x(\Lambda_k') - 2y(\Lambda_k') - e(\Lambda_k') + 2x(\Lambda_k) + 2y(\Lambda_k) + e(\Lambda_k).
\end{equation}
It follows from the necessary conditions for equality in the index inequality \eqref{eqn:indineq}, see \cite[\S3.9]{bn}, that the positive ends of $C_k$ are all at simple Reeb orbits, and $C_k$ has no two negative ends at covers of the same Reeb orbit.
Since all edges of $\Lambda_k'$ are labeled `$e$', it follows that in the notation of Proposition~\ref{prop:J0} applied to $C_k$, we have
\[
\sum_i(2n_i^+-1) = 2m(\Lambda_k') - e(\Lambda_k')
\]
and
\[
\sum_j(2n_j^--1) = e(\Lambda_k) + h(\Lambda_k).
\]
Thus Proposition~\ref{prop:J0} and the fact that $g(C_k)\ge 0$ imply that
\begin{equation}
\label{eqn:J0geom}
J_0(C_k) \ge -2 + 2m(\Lambda_k') - e(\Lambda_k') + e(\Lambda_k) + h(\Lambda_k).
\end{equation}
The inequality \eqref{eqn:lt3} then follows from \eqref{eqn:J0comb} and \eqref{eqn:J0geom}.

{\em Step 5:\/} We now complete the proof.

By Step 4, if we factor $\Lambda'$ as the product over $k$ of $\Lambda_k'$ repeated $d_k$ times, and if we factor $\Lambda$ as the product over $k$ of $\Lambda_k$ repeated $d_k$ times, then these factorizations satisfy the conclusions of Theorem~\ref{thm:convex}, except that there is a $2\epsilon$ error in the action inequality \eqref{eqn:lt2}. Here the second and third bullets in Theorem~\ref{thm:convex} follow from the corresponding bullets in Proposition~\ref{prop:key}(b).

Since $\epsilon>0$ can be arbitrarily small, and since there are only finitely many possibilities for $\Lambda$ and the factorizations, it follows by a limiting argument that there exist $\Lambda$ and factorizations which satisfy the conclusions of Theorem~\ref{thm:convex} without any error.
\end{proof}

\begin{appendix}

\section{Conjectural improvement of the main theorem}
\label{app:conjecture}

We now state a conjecture regarding the differential on the ECH chain complex for the boundary of a (perturbed) convex toric domain, which implies an improved version of Theorem~\ref{thm:convex}.

\begin{definition}
A {\em torus generator\/} is a closed convex polygon in $\R^2$ with vertices in $\Z^2$, modulo translation, such that each edge is labeled `$e$' or `$h$'.
\end{definition}

Let $C_*^{T^3}$ denote the chain complex $\overline{C_*}(2\pi;0)$ defined in \cite[\S1.3]{t3}, tensored with $\Z/2$. The chain complex $C_*^{T^3}$ is freely generated over $\Z/2$ by torus generators. The differential of a torus generator is the sum over all torus generators obtainable by ``rounding a corner'' and ``locally losing one `$h$''', see \cite[\S1.3]{t3}.

\begin{definition}
If $\Lambda$ is a convex generator, define a torus generator $\tilde{\Lambda}$ by attaching to $\Lambda$ a vertical segment from $(0,0)$ to the upper left endpoint of $\Lambda$, and a horizontal segment from $(0,0)$ to the lower right endpoint of $\Lambda$. The new horizontal and vertical segments are labeled `$e$'; all other edge labels are preserved.
\end{definition}

\begin{conjecture}
\label{conj:chainmap}
Let $X_\Omega$ be a smooth convex toric domain with boundary $Y$. Then one can choose the contact form $\lambda$ on $Y$ in Lemma~\ref{lem:perturbation}, and a generic $\lambda$-compatible almost complex structure $J$ on $\R\times Y$, so that the linear map
\begin{equation}
\label{eqn:chainmap}
ECC^L(Y,\lambda,0,J) \longrightarrow C_*^{T^3},
\end{equation}
sending the admissible orbit set $\imath(\Lambda)$ to the torus generator $\tilde{\Lambda}$, is a chain map.
\end{conjecture}

We expect that this conjecture can be proved by modifying the arguments in \cite[\S11]{t3} and quoting results of Taubes \cite{beast1,beast2}. The significance of this conjecture is as follows:

\begin{proposition}
\label{prop:significance}
Assume Conjecture~\ref{conj:chainmap}. If $\Lambda$ is any convex generator in which all edges are labeled `$e$', and if $L>A_\Omega(\Lambda)$, then $\lambda$ and $J$ in Lemma~\ref{lem:perturbation} can be chosen so that $\imath(\Lambda)$ is a cycle in the chain complex $ECC^L(Y,\lambda,0,J)$ which represents a nontrivial homology class in $ECH(Y,\lambda,0)$.
\end{proposition}

\begin{proof}
Suppose $\lambda$ and $J$ satisfy the conclusion of Conjecture~\ref{conj:chainmap}.
It follows from the definition of the differential on $C_*^{T^3}$ that $\tilde{\Lambda}$ is a cycle in $C_*^{T^3}$. Since the chain map \eqref{eqn:chainmap} is injective, it follows that $\imath(\Lambda)$ is a cycle in $ECC^L(Y,\lambda,0,J)$. It is shown in \cite[Prop.\ 8.3]{t3} that $\tilde{\Lambda}$ represents a nontrivial homology class in the homology of the chain complex $C_*^{T^3}$. Since \eqref{eqn:chainmap} is a chain map, it follows that $\imath(\Lambda)$ represents a nontrivial homology class in $ECH^L(Y,\lambda,0)$.

A direct limit argument then shows that by increasing $L$ in the above paragraph, we can ensure that $\imath(\Lambda)$ represents a nontrivial homology class in $ECH(Y,\lambda,0)$.
\end{proof}

Consequently, if Conjecture~\ref{conj:chainmap} is true, then in the statement of Theorem~\ref{thm:convex}, the assumption that $\Lambda'$ is minimal for $X_{\Omega'}$ can be weakened to the assumption that all edges of $\Lambda'$ are labeled `$e$'; because in Step 3 of the proof of Theorem~\ref{thm:convex}, one can use Proposition~\ref{prop:significance} in place of Lemma~\ref{lem:supplement} to conclude that $\imath(\Lambda')$ represents a nontrivial homology class in $ECH(Y',\lambda',0)$.

\paragraph{Acknowledgments.} I thank Keon Choi, Dan Cristofaro-Gardiner, David Frenkel, Richard Hind, Sam Lisi, Dusa McDuff, and Vinicius Ramos for helpful and inspiring discussions related to this topic. Part of this research was carried out at the Simons Center for Geometry and Physics in June 2014.

\end{appendix}

\end{document}